\documentclass[12pt,a4paper]{article} 
\usepackage{amsmath}
\usepackage{amsthm}
\usepackage{amsfonts}
\usepackage{amssymb}
\usepackage{bussproofs}
\author{Amir Akbar Tabatabai}
\theoremstyle{plain} 
\newtheorem{thm}{Theorem}[section]

\newtheorem{lem}[thm]{Lemma}

\newtheorem{cor}[thm]{Corollary}
\theoremstyle{definition}
\newtheorem{dfn}[thm]{Definition}
\newtheorem{exam}[thm]{Example}
\newtheorem{rem}[thm]{Remark}

\def\PA{\mathrm{PA}}
\def\PV{\mathrm{PV}}

\def\TI{\mathrm{TI}}
\def\BASIC{\mathrm{BASIC}}
\def\PRA{\mathrm{PRA}}
\def\EXP{\mathrm{EXP}}
\def\NP{\mathrm{NP}}

\def\TFNP{\mathrm{TFNP}}
\def\PRWO{\mathrm{PRWO}}
\def\E{\mathrm{E}}

\def\PLS{\mathrm{PLS}}

\def\GLS{\mathrm{GLS}}

\begin{document}
\title{Computational Flows in Arithmetic} 

\author{Amirhossein Akbar Tabatabai \footnote{The author is supported by the ERC Advanced Grant 339691 (FEALORA)}.\\
Institute of Mathematics\\
Academy of Sciences of the Czech Republic\\
tabatabai@math.cas.cz}

\date{\today}

\maketitle

\begin{abstract}
A computational flow is a pair consisting of a sequence of computational problems of a certain sort and a sequence of computational reductions among them. In this paper we will develop a theory for these computational flows and we will use it to make a sound and complete interpretation for bounded theories of arithmetic. This property helps us to decompose a first order arithmetical proof to a sequence of computational reductions by which we can extract the computational content of low complexity statements in some bounded theories of arithmetic such as $I\Delta_0$, $T^k_n$, $I\Delta_0+\EXP$ and $\PRA$. In the last section, by generalizing term-length flows to ordinal-length flows, we will extend our investigation from bounded theories to strong unbounded ones such as $I\Sigma_n$ and $\PA+\TI(\alpha)$ and we will capture their total $\NP$ search problems as a consequence.
\end{abstract}

\section{Introduction}
Intuitively speaking, proofs are information carriers and they transfer the informational content of the assumptions to the informational content of the conclusion. This open notion of content though admits different many interpretations in different many disciplines. The most trivial and the least informative one is the truth value of a sentence and it is pretty clear that this truth value is preserved by sound proofs. The other example, and the more useful interpretation, is \textit{the computational content} of a sentence, which plays the main role in the realm of proof theory and theoretical computer science. The notion of the computational content also admits different kinds of interpretations, from the witnesses of existential quantifiers a la Herbrand to dialectica-type interpretation of higher order arithmetical statements via G\"{o}del's type theory $T$. What we want to investigate in this paper is one of these computational interpretations and in the rest of this introduction we will try to explain it.\\

Let us explain the idea step by step. First of all, we will focus on our interpretation of the computational content of a sentence. The answer is simply the following: We will interpret a sentence as a computational problem and by its computational content we roughly mean \textit{any way} that can solve the problem computationally. It is clear that this notion of content is vague and imprecise but note that what is important is not the content itself but how it flows. (Compare this situation to the cardinal arithmetic where the notion of a cardinal of an infinite set is secondary compared to the notion of equipotency.)  Therefore, it is important to interpret the computational preservation of information and we have a very natural candidate for that: the computational reductions. Let us illuminate what we mean by an example. Consider the formula $\forall y \leq t(x) \exists z \leq s(x)A(x, y, z)$. What we mean by this sentence is the total search problem which reads $y$ in the domain $[t(x)]$ and finds $z$ in the range $[s(x)]$ such that $A(x, y, z)$ holds. This is a computational problem and by its content we mean any kind of computational method to solve this search problem. Now consider the situation that we have another search problem $\forall u \leq m(x) \exists v \leq n(x)B(x, u, v)$. The question is how it is possible to transfer the content of the first one to the content of the second one. In other words, if we have \textit{a way} to solve the first search problem, how can we find \textit{a way} to solve the second one? One of the many ways to reduce the second one to the first one is the reduction technique which we can define in the case of our example as the following: A computational reduction from $\forall u \leq m(x) \exists v \leq n(x)B(x, u, v)$ to $\forall y \leq t(x) \exists z \leq s(x)A(x, y, z)$ is the pair of two functions $f$ and $g$ with a certain complexity such that $f$ reads $u$ and finds $y=f(x, u)$ and $g$ reads $u, z$ and computes $v=g(x, u, z)$ such that 
\[
A(x, f(x, u), z) \rightarrow B(x, y, g(x, u, z)).
\]
So far, we have explained our interpretation of sentences and the way that the content is preserved. Now it is time to find a natural interpretation for proofs as the information carriers. For this goal, we will translate a first order proof of a sequent $\Gamma \Rightarrow \Delta$ to a sequence of simple provable computational reductions from $\bigwedge \Gamma$ to $\bigvee \Delta$ which formalizes the concept of a flow of computational information and for this reason we will call these sequences just computational flows or simply flows. Therefore, what we have to do is to show that this flow interpretation is sound and complete with respect to some certain theories, i.e, we have to show that if there exists a proof for $\Gamma \Rightarrow \Delta$ then there exists a flow from $\bigwedge \Gamma$ to $\bigvee \Delta$ and vice versa. This is the main goal of the whole paper. \\

As the final part of this introduction, let us say something about the structure of the paper. First of all, we will develop the theory on the abstract scale to make everything more clear and general. However, to control the problems arising from this extreme abstraction, we will limit ourselves just to the languages of arithmetic, the theories of bounded arithmetic and to some weak unbounded theories. Secondly and using these flows, we will reprove some recent characterizations of search problems in the Buss' hierarchy of bounded arithmetic via game induction principle \cite{ST}, \cite{Ta} or higher PLS problems \cite{BB} and a characterization of $\NP$ search problems of Peano Arithmetic \cite{Be}. Then we will generalize these results to prove some new characterizations of low-complexity search problems from higher order bounded theories of arithmetic and stronger theories such as $I\Delta_0+\EXP$ to strong fragments of Peano arithmetic like $I\Sigma_n$ or even stronger theories like $\PA+\TI(\alpha)$ for $\epsilon_0 \preceq \alpha$.

\section{The Theory of Flows}
In this section we will present a general definition of a bounded theory of arithmetic and then we will use two different types of flows to decompose the proofs of these theories.\\
First of all, let us fix a language which can be an arbitrary extension of a ring-type language for numbers:
\begin{dfn}\label{t2-1}
Let $\mathcal{L}$ be a first order language of arithmetic extending $\{0, 1, +, -, \cdot, \lfloor \frac{\cdot}{\cdot} \rfloor, \leq \}$. By $\mathcal{R}$ we mean the first order theory consisting of the axioms of commutative discrete ordered semirings (the usual axioms of commutative rings minus the existence of additive inverse plus the axioms to state that $\leq$ is a total discrete order such that $<$ is compatible with addition and multiplication with non-zero elements), plus the following defining axioms for $-$ and $\lfloor \frac{\cdot}{\cdot} \rfloor$:
\[
(x \geq y \rightarrow (x-y)+y=x ) \wedge
(x < y \rightarrow x-y=0),
\]
and
\[
((y+1) \cdot \lfloor \frac{x}{y} \rfloor \leq x) \wedge (x- (y+1) \cdot \lfloor \frac{x}{y} \rfloor < y+1).
\]
\end{dfn}
Note that to avoid division by zero and to have a total function symbol in the language, by $\lfloor \frac{x}{y} \rfloor$ we actually mean $\lfloor \frac{x}{y+1} \rfloor$.
\begin{dfn}\label{t2-2}
Let $\mathcal{B}$ be a theory. A class of terms, $\mathbb{T}$, is called a $\mathcal{B}$-term set if:
\begin{itemize}
\item[$(i)$]
It is closed under all $\mathcal{L}_{\mathcal{R}}$-basic term operations of the language $\mathcal{L}$ provably in $\mathcal{B}$, i.e. for any basic operation $f$ and any $t(\vec{x}) \in \mathbb{T}$, there exist $r(\vec{x}) \in \mathbb{T}$ such that $\mathcal{B} \vdash r(\vec{x})=f(t(\vec{x}))$.
\item[$(ii)$]
It is closed under substitution, i.e. if $t(\vec{x}, y) \in \mathbb{T}$ and $s$ is an arbitrary term (not necessarily in $\mathbb{T}$ ) then $t(\vec{x}, s) \in \mathbb{T}$ provably in $\mathcal{B}$, i.e. there exists $r(\vec{x}) \in \mathbb{T}$ such that $\mathcal{B} \vdash r(\vec{x})=t(\vec{x}, s)$.
\end{itemize}
Moreover, if a term set $\mathbb{T}$ has a subset of monotone majorizing terms provably in $\mathcal{B}$, it is called a $\mathcal{B}$-term ideal. By a monotone majorizing subset we mean a set of terms $X \subseteq \mathbb{T}$ such that for any $t(\vec{x}) \in \mathbb{T}$ there exists $s(\vec{x}) \in X$ such that $\mathcal{B} \vdash t(\vec{x}) \leq s(\vec{x})$ and for any $r(\vec{x}) \in X$, $\mathcal{B} \vdash \vec{x} \leq \vec{y} \rightarrow r(\vec{x}) \leq r(\vec{x})$.
\end{dfn}
\begin{exam}\label{t2-3}
For any theory $\mathcal{B}$ and any language extending $\mathcal{L}_{\mathcal{R}}$, there are two trivial $\mathcal{B}$-term sets. $\mathbb{T}_{all}$ consisting of all terms of the language and $\mathbb{T}_{cls}$ consisting of all closed terms. To have a non-trivial example, consider the language of bounded arithmetic extending the language of $\mathcal{R}$ and define $\mathbb{T}_{p}$ as the class of all terms majorized by a term in the form $p(|\vec{x}|)$ for some polynomial $p$ provably in $\BASIC+\mathcal{R}$. The majorzing subset is the set of all terms in the form $p(|\vec{x}|)$.
\end{exam}
\begin{dfn}\label{t2-4}
\begin{itemize}
\item[$(i)$]
By an $\mathcal{R}$-conjunction between $A(\vec{x})$ and $B(\vec{x})$ we mean a formula $C(\vec{x})$ such that $\mathcal{R} \vdash A(\vec{x}) \wedge B(\vec{x}) \leftrightarrow C(\vec{x})$.
\item[$(ii)$]
By an $\mathcal{R}$-disjunction between $A(\vec{x})$ and $B(\vec{x})$ we mean a formula $C(\vec{x})$ such that $\mathcal{R} \vdash A(\vec{x}) \vee B(\vec{x}) \leftrightarrow C(\vec{x})$.
\item[$(iii)$]
By an $\mathcal{R}$-negation for $A(\vec{x})$ we mean a formula $C(\vec{x})$ such that $\mathcal{R} \vdash \neg A(\vec{x}) \leftrightarrow C(\vec{x})$.
\item[$(iv)$]
By an $\mathcal{R}$-bounded universal quantification for $A(\vec{x}, y)$ we mean a formula $C(\vec{x})$ such that $\mathcal{R} \vdash \forall y \leq t(\vec{x}) A(\vec{x}, y) \leftrightarrow C(\vec{x})$.
\item[$(v)$]
By an $\mathcal{R}$-bounded existential quantification for $A(\vec{x}, y)$ we mean a formula $C(\vec{x})$ such that $\mathcal{R} \vdash \exists y \leq t(\vec{x}) A(\vec{x}, y) \leftrightarrow C(\vec{x})$.
\end{itemize}
\end{dfn}
Using the general setting we have set so far we can also define a general definition of $\pi$ and $\sigma$-classes.
\begin{dfn}\label{t2-5}
\begin{itemize}
\item[$(i)$]
A class of formulas $\Pi$ is called a $\pi$-class of the language $\mathcal{L}$ if it includes all quantifier-free formulas of $\mathcal{L}$, is closed under substitutions and subformulas, and is closed under an $\mathcal{R}$-conjunction, an $\mathcal{R}$-disjunction and an $\mathcal{R}$-bounded universal quantifier. And finally, if $\exists y \leq t \; B(y) \in \Pi$ then $B$ has an $\mathcal{R}$-negation in $\Pi$ and also $\forall y \leq t \; \neg B(y)$ has an $\mathcal{R}$-negation in $\Pi$, i.e $\neg \forall y \leq t  \; \neg B(y) \in \Pi$.
\item[$(ii)$]
A class of formulas $\Sigma$ is called a $\sigma$-class of the language $\mathcal{L}$ if it includes all quantifier-free formulas of $\mathcal{L}$, is closed under substitutions and subformulas, is closed under an $\mathcal{R}$-conjunction, an $\mathcal{R}$-disjunction and an $\mathcal{R}$-bounded existential quantifier. And finally, if $\forall y \leq t \; B(y) \in \Sigma$ then $B$ has an $\mathcal{R}$-negation in $\Sigma$ and also $\exists y \leq t \neg \; B(y)$ has an $\mathcal{R}$-negation in $\Sigma$, i.e $\neg \exists y \leq t \; \neg B(y) \in \Sigma$.
\end{itemize}
\end{dfn}
We can also define a bounded hierarchy:
\begin{dfn}\label{t2-6}
Let $\Phi$ be a class that includes all quantifier-free formulas and is closed under all boolean operations. The hierarchy $\{\Sigma_k(\Phi), \Pi_k(\Phi) \}_{k=0}^{\infty}$ is defined as the following:
\begin{itemize}
\item[$(i)$]
$\Pi_0(\Phi)=\Sigma_0(\Phi)$ is the class $\Phi$,
\item[$(ii)$]
If $B(x) \in \Sigma_k(\Phi)$ then $\exists x \leq t \; B(x) \in \Sigma_k(\Phi)$ and $\forall x \leq t \; B(x) \in \Pi_{k+1}(\Phi)$ and
\item[$(iii)$]
If $B(x) \in \Pi_k(\Phi)$ then $\forall x \leq t \; B(x) \in \Pi_k(\Phi)$ and $\forall x \leq t \; B(x) \in \Sigma_{k+1}(\Phi)$.
\end{itemize}
\end{dfn}
\begin{exam}\label{t2-7}
The most well-known examples of $\pi$ and $\sigma$-classes are $U^k$ and $\E^k$ classes in the language of Peano arithmetic, $\Pi^b_k$ and $\Sigma^b_k$ classes and $\hat{\Pi}^b_k$ and $\hat{\Sigma}^b_k$ classes in the language of bounded arithmetic. But there are also some other useful examples, like the classes based on doubly sharply bounded formulas following with alternating sharply bounded quantifiers in the language of bounded arithmetic plus the function $\#_3$.
\end{exam}
We are ready to state the general definition of bounded arithmetic.
\begin{dfn}\label{t2-8}
Let $\mathcal{A}$ be a set of quantifier-free axioms, $\mathbb{T}$ be a $\mathcal{A}$-term ideal and $\Phi$ be a class of bounded formulas closed under substitution and subformulas. By the first order bounded arithmetic, $\mathfrak{B}(\mathbb{T}, \Phi, \mathcal{A})$ we mean the theory in the language $\mathcal{L}$ which consists of axioms $\mathcal{A}$, and the $(\mathbb{T}, \Phi)$-induction axiom, i.e. ,
\[
A(0) \wedge \forall x (A(x) \rightarrow A(x+1)) \rightarrow \forall x A(t(x)),
\]
where $A \in \Phi$ and $t \in \mathbb{T}$. Equivalently, we can define $\mathfrak{B}(\mathbb{T}, \Phi, \mathcal{A})$ as a proof system of the following form:
\begin{flushleft}
	\textbf{Axioms:}
\end{flushleft}
\begin{center}
	\begin{tabular}{c c c}
		\AxiomC{ }
		\UnaryInfC{$  A \Rightarrow A$}
		\DisplayProof 
			&\;\;\;\;
		\AxiomC{ }
		\UnaryInfC{$  \bot \Rightarrow $}
		\DisplayProof 
		    &\;\;\;\;
	    \AxiomC{ }
		\LeftLabel{}
		\UnaryInfC{$  \Rightarrow A$}
		\DisplayProof 
	\end{tabular}
\end{center}
Where in the rightmost rule, $A \in \mathcal{A}$.
\begin{flushleft}
 		\textbf{Structural Rules:}
\end{flushleft}
\begin{center}
	\begin{tabular}{c}
		\begin{tabular}{c c}
		\AxiomC{$\Gamma  \Rightarrow \Delta$}
		\LeftLabel{\tiny{$ (wL) $}}
		\UnaryInfC{$\Gamma,  A  \Rightarrow \Delta$}
		\DisplayProof
			&
		\AxiomC{$\Gamma  \Rightarrow \Delta$}
		\LeftLabel{\tiny{$ ( wR) $}}
		\UnaryInfC{$\Gamma \Rightarrow  \Delta, A$}
		\DisplayProof
		\end{tabular}
			\\[3 ex]
			\begin{tabular}{c c}
		\AxiomC{$\Gamma, A, A \Rightarrow \Delta$}
		\LeftLabel{\tiny{$ (cL) $}}
		\UnaryInfC{$\Gamma,  A  \Rightarrow \Delta$}
		\DisplayProof
		    &
		\AxiomC{$\Gamma \Rightarrow \Delta, A, A$}
		\LeftLabel{\tiny{$ (cR) $}}
		\UnaryInfC{$\Gamma \Rightarrow \Delta, A$}
		\DisplayProof
		\end{tabular}
        \\[3 ex]
	    
	    \AxiomC{$\Gamma_0 \Rightarrow \Delta_0, A$}
	    \AxiomC{$\Gamma_1, A \Rightarrow \Delta_1$}
		\LeftLabel{\tiny{$ (cut) $}}
		\BinaryInfC{$\Gamma_0, \Gamma_1 \Rightarrow \Delta_0, \Delta_1$}
		\DisplayProof
	\end{tabular}
\end{center}		
\begin{flushleft}
  		\textbf{Propositional Rules:}
\end{flushleft}
\begin{center}
  	\begin{tabular}{c c}
  		\AxiomC{$\Gamma_0, A  \Rightarrow \Delta_0 $}
  		\AxiomC{$\Gamma_1, B  \Rightarrow \Delta_1$}
  		\LeftLabel{{\tiny $\vee L$}} 
  		\BinaryInfC{$ \Gamma_0, \Gamma_1, A \lor B \Rightarrow \Delta_0, \Delta_1 $}
  		\DisplayProof
	  		&
	   	\AxiomC{$ \Gamma \Rightarrow \Delta, A_i$}
   		\RightLabel{{\tiny $ (i=0, 1) $}}
   		\LeftLabel{{\tiny $\vee R$}} 
   		\UnaryInfC{$ \Gamma \Rightarrow \Delta, A_0 \lor A_1$}
   		\DisplayProof
	   		\\[3 ex]
   		\AxiomC{$ \Gamma, A_i \Rightarrow \Delta$}
   		\RightLabel{{\tiny $ (i=0, 1) $}} 
   		\LeftLabel{{\tiny $\wedge L$}}  		
   		\UnaryInfC{$ \Gamma, A_0 \land A_1 \Rightarrow \Delta, C $}
   		\DisplayProof
	   		&
   		\AxiomC{$\Gamma_0  \Rightarrow \Delta_0, A$}
   		\AxiomC{$\Gamma_1  \Rightarrow  \Delta_1, B$}
   		\LeftLabel{{\tiny $\wedge R$}} 
   		\BinaryInfC{$ \Gamma_0, \Gamma_1 \Rightarrow \Delta_0, \Delta_1, A \land B $}
   		\DisplayProof
   			\\[3 ex]
   		\AxiomC{$ \Gamma_0 \Rightarrow A, \Delta_0 $}
  		\AxiomC{$ \Gamma_1, B \Rightarrow \Delta_1, C $}
  		\LeftLabel{{\tiny $\rightarrow L$}} 
   		\BinaryInfC{$ \Gamma_0, \Gamma_1, A \rightarrow B \Rightarrow \Delta_0, \Delta_1, C$}
   		\DisplayProof
   			&
   		\AxiomC{$ \Gamma, A \Rightarrow B, \Delta $}
   		\LeftLabel{{\tiny $\rightarrow R$}} 
   		\UnaryInfC{$ \Gamma \Rightarrow \Delta, A \rightarrow B$}
   		\DisplayProof
   		\\[3 ex]
   		\AxiomC{$ \Gamma \Rightarrow \Delta, A $}
   		\LeftLabel{\tiny {$\neg L$}} 
   		\UnaryInfC{$ \Gamma, \neg A \Rightarrow \Delta$}
   		\DisplayProof
   			&
   		\AxiomC{$ \Gamma, A \Rightarrow \Delta $}
   		\LeftLabel{{\tiny $\neg R$}} 
   		\UnaryInfC{$ \Gamma \Rightarrow \Delta, \neg A$}
   		\DisplayProof
	\end{tabular}
\end{center}
\begin{flushleft}
  		\textbf{Quantifier rules:}
\end{flushleft}
\begin{center}
  	\begin{tabular}{c c}
  		\AxiomC{$\Gamma, A(s)  \Rightarrow \Delta $}
  		\LeftLabel{{\tiny $\forall L$}} 
  		\UnaryInfC{$ \Gamma, \forall y \; A(y) \Rightarrow \Delta $}
  		\DisplayProof
	  		&
	   	\AxiomC{$\Gamma  \Rightarrow \Delta, A(y) $}
  		\LeftLabel{{\tiny $\forall R$}} 
  		\UnaryInfC{$ \Gamma \Rightarrow \Delta, \forall y \; A(y) $}
  		\DisplayProof
	   		\\[3 ex]
   		\AxiomC{$\Gamma, A(y) \Rightarrow \Delta $}
  		\LeftLabel{{\tiny $\exists L$}} 
  		\UnaryInfC{$ \Gamma, \exists y \; A(y) \Rightarrow \Delta $}
  		\DisplayProof
	   		&
   		\AxiomC{$\Gamma \Rightarrow \Delta, A(s)  $}
  		\LeftLabel{{\tiny $\exists R$}} 
  		\UnaryInfC{$ \Gamma, \Rightarrow \Delta, \exists y \; A(y) $}
  		\DisplayProof
   			\\[3 ex]
	\end{tabular}
\end{center}
\begin{flushleft}
  		\textbf{Bounded Quantifier rules:}
\end{flushleft}
\begin{center}
  	\begin{tabular}{c c}
  		\AxiomC{$\Gamma, A(s)  \Rightarrow \Delta $}
  		\LeftLabel{{\tiny $\forall^{\leq} L$}} 
  		\UnaryInfC{$ \Gamma, s \leq t, \forall y \leq t \; A(y) \Rightarrow \Delta $}
  		\DisplayProof
	  		&
	   	\AxiomC{$\Gamma, y \leq t  \Rightarrow \Delta, A(y) $}
  		\LeftLabel{{\tiny $\forall^{\leq} R$}} 
  		\UnaryInfC{$ \Gamma \Rightarrow \Delta, \forall y \leq t \; A(y) $}
  		\DisplayProof
	   		\\[3 ex]
   		\AxiomC{$\Gamma, y \leq t, A(y) \Rightarrow \Delta $}
  		\LeftLabel{{\tiny $\exists^{\leq} L$}} 
  		\UnaryInfC{$ \Gamma, \exists y \leq t \; A(y) \Rightarrow \Delta $}
  		\DisplayProof
	   		&
   		\AxiomC{$\Gamma \Rightarrow \Delta, A(s)  $}
  		\LeftLabel{{\tiny $\exists^{\leq} R$}} 
  		\UnaryInfC{$ \Gamma, s \leq t, \Rightarrow \Delta, \exists y \leq t \; A(y) $}
  		\DisplayProof
   			\\[3 ex]
	\end{tabular}
\end{center}
And the following induction rule:
\begin{flushleft}
 		\textbf{Induction:}
\end{flushleft}
\begin{center}
	\begin{tabular}{c}
	    \AxiomC{$\Gamma, A(y) \Rightarrow \Delta, A(y+1)$}
		\LeftLabel{\tiny{$ (Ind) $}}
		\UnaryInfC{$\Gamma, A(0) \Rightarrow \Delta, A(t)$}
		\DisplayProof
	\end{tabular}
\end{center}		
For every $A \in \Phi$ and $t \in \mathbb{T}$.
\end{dfn}
\begin{exam}\label{t2-9}
With our definition of bounded arithmetic, different kinds of theories can be considered as bounded theories of arithmetic, for instance $I\Delta_0$, $S_i^k$, $T^k_i$, $I\Delta_0+\EXP$ and $\PRA$ are just some of the well-known examples.
\end{exam}
The most important property of the sequent calculus of bounded theories of arithmetic is cut elimination:
\begin{thm}\label{t2-10}(Cut Elimination)
If $\mathfrak{B}(\mathbb{T}, \Phi, \mathcal{A}) \vdash \Gamma \Rightarrow \Delta$ then there exists a free-cut free proof for the same sequent in the same system.
\end{thm}
The following corollary is very useful:
\begin{cor}\label{t2-11}
If $\Gamma \cup \Delta \subseteq \Phi$ and $\mathfrak{B}(\mathbb{T}, \Phi, \mathcal{A}) \vdash \Gamma \Rightarrow \Delta$ then there exists a proof of the same sequent in the same sytem such that all formulas occurring in the proof is in $\Phi$. 
\end{cor}
In the following we will define two different types of reductions as the building blocks of flows. These reductions are generalizations of the usual reductions in computablity theory, from many to one reductions between recursive languages to polytime reductions between total $\NP$ search problems. 

\begin{dfn}\label{t2-12}
Let $A(\vec{x})$ and $B(\vec{x})$ be some formulas in $\Pi_k(\Phi)$ and $\{F_i\}_{i=1}^k$ be a sequence of terms. By recursion on $k$, we will define $F=\{F_i\}_{i=1}^k$ as a deterministic $\Pi_k(\Phi)$-reduction from $B(\vec{x})$ to $A(\vec{x})$ and we will denote it by $A(\vec{x}) \leq_{d}^{F, k} B(\vec{x})$ when:
\begin{itemize}
\item[$(i)$]
If $A(\vec{x}), B(\vec{x})$ are in $\Pi_0(\Phi)$, we say that the empty sequence of functions is a deterministic reduction from $B$ to $A$ iff $\mathcal{B} \vdash A(\vec{x}) \rightarrow B(\vec{x})$.
\item[$(ii)$]
If $A=\forall \vec{u} \leq \vec{p}(\vec{x}) C(\vec{x}, \vec{u})$, $B=\forall \vec{v} \leq \vec{q}(\vec{x}) D(\vec{x}, \vec{v})$ and $F=\{F_{i} \}_{i=1}^{k+1}$ is a sequence of terms, then $A(\vec{x}) \leq_{d}^{F, k+1} B(\vec{x})$ iff
\[
\mathcal{B} \vdash \vec{v} \leq \vec{q}(\vec{x}) \rightarrow F_{k+1}(\vec{x}, \vec{v}) \leq \vec{p}(\vec{x})
\]
and
\[
F_{k+1}(\vec{x}, \vec{v}) \leq \vec{p}(\vec{x}) \rightarrow C(\vec{x}, F_{k+1}(\vec{x}, \vec{v})) \leq_{d}^{\hat{F}, k} \vec{v} \leq \vec{q}(\vec{x}) \rightarrow D(\vec{x}, \vec{v})
\]
where $\hat{F}=\{F_{i}\}_{i=1}^k$.
\item[$(iii)$]
If $A=\exists \vec{u} \leq \vec{p}(\vec{x}) C(\vec{x}, \vec{u})$, $B=\exists \vec{v} \leq \vec{q}(\vec{x}) D(\vec{x}, \vec{v})$ and $F=\{F_{i} \}_{i=1}^{k+1}$ is a sequence of terms, then $A(\vec{x}) \leq_{d}^{F, k+1} B(\vec{x})$ iff
\[
\mathcal{B} \vdash \vec{u} \leq \vec{p}(\vec{x}) \rightarrow F_{k+1}(\vec{x}, \vec{u}) \leq \vec{q}(\vec{x})
\]
and
\[
\vec{y} \leq \vec{p}(\vec{x}) \wedge C(\vec{x}, u) \leq_{d}^{\hat{F}, k}  F_{k+1}(\vec{x}, \vec{u}) \leq \vec{q}(\vec{x}) \wedge D(\vec{x}, F_{k+1}(\vec{x}, \vec{u}))
\]
where $\hat{F}=\{F_{i}\}_{i=1}^k$.
\end{itemize}
We say $B$ is $(\Pi_k(\Phi), \mathcal{B})$-deterministicly reducible to $A$ and we write $A \leq_d^{(\Pi_k(\Phi), \mathcal{B})} B$, when there exists a sequence of terms $F$ such that $A \leq^{F, k}_{d} B$.
\end{dfn}

\begin{dfn}\label{t2-13}
Let $\mathcal{B}$ be a first order bounded arithmetic and $A(\vec{x})$ and $B(\vec{x})$ be some formulas in the language $\mathcal{L}$. We say $B$ is non-deterministically $\mathcal{B}$-reducible to $A(\vec{x})$ and we write $A(\vec{x}) \leq_{n}^{\mathcal{B}} B(\vec{x}) $ if $\mathcal{B} \vdash A(\vec{x}) \rightarrow B(\vec{x})$. 
\end{dfn}
The natural question is that how this proof-theoretic based concept can be called a computational reduction and if so, why is it a non-deterministic reduction as opposed to the above-mentioned deterministic reduction? The answer is the following well-known Herbrand theorem:
\begin{thm}\label{t2-14}(Herbrand Theorem)
If $\mathcal{B}$ is a universal bounded arithmetic then the following are equivalent:
\begin{itemize}
\item[$(i)$]
$A(\vec{x}) \leq_{n}^{\mathcal{B}} B(\vec{x})$.
\item[$(ii)$]
There exists a Herbrand proof for $A(\vec{x}) \rightarrow B(\vec{x})$ in $\mathcal{B}$.
\end{itemize}
\end{thm}
Generally speaking, we intend to decompose arithmetical proofs to a sequence of reductions, and the base theory for those reductions preferably are simple universal and possibly induction-free theories. Therefore, we can use the Herbrand theorem for each step of the reduction to witness the essentially existential quantifiers in $A \rightarrow B$. This is actually what is happening in the deterministic reductions, but here the difference is the use of $\vee$-expansions in the Herbrand proof. Intuitively, these expansions allow us to use some constantly many terms to witness one existential quantifier as opposed to just one term in the case of deterministic reductions. Moreover, expansions make some room for interaction in providing the witnessing terms which makes the concrete witnesses extremely complicated. For these reasons, we call these reductions non-deterministic. \\
 
In the following examples we will illuminate the difference between deterministic and non-deterministic reductions and the importance and the naturalness of the latter.
\begin{exam}\label{t2-15}
Let $A(x, y) \in \Pi^b_k$ be a formula and consider the sentences $ \exists y, z \leq t \; (A(x, y) \vee  A(x, z))$ and $ \exists w \leq t \; A(x, w)$. Intuitively, the first formula is equivalent to $ \exists y \leq t \; A(x, y) \vee \exists z \leq t \; A(x, z)$ which is equivalent to the the second formula $\exists w \leq t \; A(x, w)$. Therefore, it seems quite reasonable to assume that if we have the second one, we can reduce the first one to it. Moreover, since this equivalence is quite elementary and it is just on the level of pure first order logic, we expect the reduction to have the lowest possible complexity. Fortunately, for the non-deterministic reduction it is obviously the case. But let us try to understand how the computational aspect of this reduction works. To do so, we have to take a look at a proof of the statement $\exists y, z \leq t \; (A(x, y) \vee  A(x, z)) \rightarrow \exists w \leq t \; A(x, w)$. The most simple proof works as follows: Assume $y$ and $z$ such that $A(x, y) \vee  A(x, z)$. Then there are two possibilities: If $A(x, y)$ then pick $w=y$ and if $A(x, z)$ then pick $w=z$. In a more computational interpretation, if we define $g(x, y, z)=y$, $h(x, y, z)=z$ we have
\[
A(x, y) \vee  A(x, z) \leq_{d} A(x, g(x, y, z)) \vee A(x, h(x, y, z)).
\]
What does it mean? It simply means that to have a reduction from the second statement to the first one we need two different copies of $\exists w \leq t A(x, y)$; one to handle the case $A(x, y)$ and the other to handle the case $A(x, z)$. This is available in proof theory via the contraction rule and it is absent in the computational interpretations of reduction. To fill this gap we allow these different copies which can be considered as some kind of non-determinism.
\end{exam}
\begin{exam}\label{t2-16}
In this example we want to show that it is generally impossible to simulate the non-deterministic reductions by deterministic ones. For that reason, we use a special case of the last example. Assume $A(x, y, z, t)=(y=0 \wedge B(x, t)) \vee (y=1 \wedge \neg B(x, z))$ where $B(x, t) \in \Pi_0^b$ is an arbitrary formula and the language consists of all polynomial computable functions ($\mathcal{L}_{\PV}$). We want to show that there is no polynomial time computable reduction from
\[
\exists y, y' \leq 1 \exists t, t' \leq s \forall z, z' \leq s \; (A(x, y, z, t) \vee  A(x, y', z', t'))  
\]
to
\[
\exists u \leq 1 \exists v \leq s \forall w \leq s \; A(x, u, v, w)
\]
even if we assume $\mathcal{B}=Th(\mathbb{N})$. Assume that there exists a polytime reduction, hence there exist a polytime function $f$ such that:
\[
\exists t, t' \leq s \forall z, z' \leq s \; (A(x, y, z, t) \vee  A(x', y', z', t'))  
\]
implies 
\[
\exists v \leq s \forall w \leq s \; A(x, f(x, y, y'), v, w).
\]
Pick $y=0$ and $y'=1$. It is easy to see that the left side is true because either $\exists t \leq s \; B(x, t)$ or $\forall z \leq s \; \neg B(x, z)$ is true, hence the right side should be true, as well. But the truth of the right side means
\[
(f(x, 0, 1)=0 \wedge \exists v \leq s \; B(x, v)) \vee (f(x, 0, 1)=1 \wedge \forall w \leq s \; \neg B(x, w))
\]
which means that we have a polytime decision procedure for the $\NP$ predicate $\exists w \leq s \; B(x, v)$ which implies $\mathbf{NP}=\mathbf{P}$. 
\end{exam}
\begin{rem}\label{t2-17}
The example \ref{t2-16} shows that pure logical deductions are far beyond the power of deterministic reductions. In other words, it is possible to prove $B$ by $A$ just by some elementary methods of logic but it does not mean that $B$ can be reducible to $A$. Let us explain where the problem is. At the first glance, it seems that all logical rules are completely syntactical and amenable to low complexity reductions. It is correct everywhere except for one logical rule: The contraction rule which is more or less responsible for all kinds of computational explosions like the explosion of the lengths of the proofs after the elimination of cuts. Notice that the reason that we have the equivalence in the Example \ref{t2-15} is this contraction rule and it is easy to see that this rule is a source of some non-determinism. Therefore, it seems natural to use non-deterministic reductions to simulate computationally what is going on in the realm of proofs.   
\end{rem}
So far, we have defined the concept of reduction which can be considered as a way to transfer the computational content of the source to the content of the target. They are similar to simple syntactic rules in the general proof theory. Then what is the counterpart of the concept of a proof (which is basically a combination of those simple rules)? The answer is the notion of a flow; a sequence of reductions which allows us to transfer information and computational contents.
\begin{dfn}\label{t2-18}
Let $\Pi$ be a $\pi$-class, $A(\vec{x}), B(\vec{x}) \in \Pi$ and $\mathbb{T}$ a term ideal. A non-deterministic $(\mathbb{T}, \Pi, \mathcal{B})$-flow from $A(\vec{x})$ to $B(\vec{x})$ is a pair $(t, H)$ where $t(\vec{x}) \in \mathbb{T}$ is a term and $H(u, \vec{x}) \in \Pi$ is a formula such that the following statements are provable in $\mathcal{B}$:
\begin{itemize}
\item[$(i)$]
$H(0, \vec{x}) \leftrightarrow A(\vec{x})$.
\item[$(ii)$]
$ H(t(x), \vec{x}) \leftrightarrow B(\vec{x})$.
\item[$(iii)$]
$\forall u < t(x) \; H(u, \vec{x}) \rightarrow H(u+1, \vec{x})$.
\end{itemize}
If there exists a non-deterministic $(\mathbb{T}, \Pi, \mathcal{B})$-flow from $A(\vec{x})$ to $B(\vec{x})$ we will write $A(\vec{x}) \rhd_{n}^{(\mathbb{T}, \Pi, \mathcal{B})} B(\vec{x})$. Moreover, if $\Gamma$ and $\Delta$ are sequents of formulas in $\Pi$, by $\Gamma \rhd_n^{(\mathbb{T}, \Pi, \mathcal{B})} \Delta$ we mean $\bigwedge \Gamma \rhd_n^{(\mathbb{T}, \Pi, \mathcal{B})} \bigvee \Delta$.
\end{dfn}
And also we have deterministic flows:
\begin{dfn}\label{t2-19}
Let $A(\vec{x}), B(\vec{x}) \in \Pi_k(\Phi)$. A $(\Pi_k(\Phi), \mathcal{B})$-deterministic flow from $A(\vec{x})$ to $B(\vec{x})$ is the following data: A term $t(\vec{x})$, a formula $H(u, \vec{x}) \in \Pi_k(\Phi)$ and sequences of terms $E_0$, $E_1$, $G_0$, $G_1$ and $F(u)$ such that the following statements are provable in $\mathcal{B}$:
\begin{itemize}
\item[$(i)$]
$H(0, \vec{x}) \equiv_{d}^{(E_0, E_1)} A(\vec{x})$.
\item[$(ii)$]
$H(t(x), \vec{x})\equiv_{d}^{(G_0, G_1)} B(\vec{x})$.
\item[$(iii)$]
$\forall u < t(x) H(u, \vec{x}) \leq_{d}^{F(u)} H(u+1, \vec{x})$.
\end{itemize}
If there exists a deterministic $(\Pi_k(\Phi), \mathcal{B})$-flow from $A(\vec{x})$ to $B(\vec{x})$ we will write $A(\vec{x}) \rhd_d^{(\Pi_k(\Phi), \mathcal{B})} B(\vec{x})$.  Moreover, if $\Gamma$ and $\Delta$ are sequents of formulas in $\Pi_k(\Phi)$, by $\Gamma \rhd_d^{(\Pi_k(\Phi), \mathcal{B})} \Delta$ we mean $\bigwedge \Gamma \rhd_d^{(\Pi_k(\Phi), \mathcal{B})} \bigvee \Delta$.
\end{dfn}

In the following we will prove a sequence of lemmas to make a high-level calculus of deterministic and non-deterministic flows. Then we will use this calculus to show that this flow interpretation is sound and complete with respect to the corresponding bounded arithmetic. All lemmas are true both for deterministic and non-deterministic flows, but note that for the deterministic flows we mean the $(\mathbb{T}, \Pi_k(\Phi), \mathcal{B})$-flow and for the non-deterministic case we mean the $(\Pi, \mathcal{B})$-flow all the time. Therefore, when we write $A \in \Pi$, for the deterministic case we mean $A \in \Pi_k(\Phi)$ and when we write $\rhd$ we mean both deterministic and non-deterministic cases.
\begin{lem}\label{t2-21}(Conjunction Application)
Let $C(\vec{x}) \in \Pi$ be a formula. If $A(\vec{x}) \rhd B(\vec{x}) $ then $A(\vec{x}) \wedge C(\vec{x}) \rhd B(\vec{x}) \wedge C(\vec{x}) $.
\end{lem}
\begin{proof}
For the deterministic case, since $A(\vec{x}) \rhd_d B(\vec{x}) $, by Definition \ref{t2-19}, there exists a term $t(\vec{x})$, a formula $H(u, \vec{x}) \in \Pi_k(\Phi)$ and sequences of terms $E_0$, $E_1$, $G_0$, $G_1$ and $F(u)$ such that 
\[
\mathcal{B} \vdash A(\vec{x}) \equiv^{E_0, E_1} H(0, \vec{x}),
\]
\[
\mathcal{B} \vdash B(\vec{x}) \equiv^{G_0, G_1} H(t(\vec{x}), \vec{x}),
\]
and
\[
\mathcal{B} \vdash \forall u < t(\vec{x}) \; H(u, \vec{x}) \leq_d^{F(u)} H(u+1, \vec{x}).
\]
Now define $t'=t$, $H'(u, \vec{x})=H(u, \vec{x}) \wedge C(\vec{x})$ and $E'_0$, $E'_1$, $G'_0$, $G'_1$ and $F'(u)$ as the corresponding sequences of terms extending their counterparts by using the quantifiers in $C$ to witness themselves by the identity terms. It is clear that the new data is a deterministic $(\Pi_k(\Phi), \mathcal{B})$-flow from $A(\vec{x}) \wedge C(\vec{x})$ to $B(\vec{x}) \wedge C(\vec{x})$. \\

For the non-deterministic case do the same, without the sequences of the terms and use the fact that if 
\[
\mathcal{B} \vdash H(u, \vec{x}) \rightarrow H(u+1, \vec{x})
\]
then,
\[
\mathcal{B} \vdash H(u, \vec{x}) \wedge C(\vec{x}) \rightarrow H(u+1, \vec{x}) \wedge C(\vec{x}).
\]  
\end{proof}
\begin{lem}\label{t2-22}(Disjunction Application)
Let $C(\vec{x}) \in \Pi$ be a formula. If $A(\vec{x}) \rhd B(\vec{x}) $ then $A(\vec{x}) \vee C(\vec{x}) \rhd B(\vec{x}) \vee C(\vec{x}) $.
\end{lem}
\begin{proof}
For the deterministic case, since $A(\vec{x}) \rhd B(\vec{x}) $ then by Definition \ref{t2-19}, there exists a term $t(\vec{x})$, a formula $H(u, \vec{x}) \in \Pi_k(\Phi)$ and sequences of terms $E_0$, $E_1$, $G_0$, $G_1$ and $F(u)$ such that the conditions in the Definition \ref{t2-19} is provable in $\mathcal{B}$. Now define $t'=t$, $H'(u, \vec{x})=H(u, \vec{x}) \vee C(\vec{x})$ and $E'_0$, $E'_1$, $G'_0$, $G'_1$ and $F'(u)$ as the corresponding sequences of terms extending their counterparts by using the quantifiers in $C$ to witness themselves by the identity terms. It is clear that the new data is a deterministic $(\Pi_k(\Phi), \mathcal{B})$-flow from $A(\vec{x}) \vee C(\vec{x})$ to $B(\vec{x}) \vee C(\vec{x})$. \\

For the non-deterministic case do the same, without the sequences of the terms and use the fact that if 
\[
\mathcal{B} \vdash H(u, \vec{x}) \rightarrow H(u+1, \vec{x})
\]
then, 
\[
\mathcal{B} \vdash H(u, \vec{x}) \vee C(\vec{x}) \rightarrow H(u+1, \vec{x}) \vee C(\vec{x}).
\]   
\end{proof}
\begin{lem}\label{t2-24}
\begin{itemize}
\item[$(i)$](Weak Gluing)
If $A(\vec{x}) \rhd B(\vec{x}) $ and $ B(\vec{x}) \rhd C(\vec{x})$ then $A(\vec{x}) \rhd C(\vec{x})$.
\item[$(ii)$](Strong Gluing)
If $s \in \mathbb{T}$ and $ A(y, \vec{x}) \rhd A(y+1, \vec{x})$ then $ A(0, \vec{x}) \rhd  A(s, \vec{x})$.
\end{itemize}
\end{lem}
\begin{proof}
For $(i)$ and for the deterministic case, since $A(\vec{x}) \rhd_d B(\vec{x})$ there exists a term $t(\vec{x})$, a formula $H(u, \vec{x}) \in \Pi_k(\Phi)$ and sequences of terms $E_0$, $E_1$, $G_0$, $G_1$ and $F(u)$ such that $\mathcal{B}$ proves the conditions in the Definition \ref{t2-19}. On the other hand since $B(\vec{x}) \rhd_d C(\vec{x})$ we have the corresponding data for $B(\vec{x})$ to $C(\vec{x})$ which we show by $t'(\vec{x})$, $H'(u, \vec{x})$, $E'_0$, $E'_1$, $G'_0$, $G'_1$ and $F'(u)$. Define $s(\vec{x})=t(\vec{x})+t'(\vec{x})+1$, 
\[
I(u, \vec{x})=
\begin{cases}
H(u, \vec{x}) & u \leq t(\vec{x})\\
B(\vec{x}) & u= t(\vec{x})+1\\
H'(u-t(\vec{x})-2, \vec{x}) & t(\vec{x})+1<u \leq t(\vec{x})+t'(\vec{x})+1
\end{cases} 
\]
and the sequence of terms in the same pointwise way. Then, it is easy to check that this new data is a deterministic $(\Pi_k(\Phi), \mathcal{B})$-flow from $A(\vec{x})$ to $C(\vec{x})$. For the non-deterministic case do the same without the sequences of the terms and notice that since $\mathbb{T}$ is closed under successor and addition and $t, t' \in \mathbb{T}$, we have $s \in \mathbb{T}$. \\

For $(ii)$ and for the deterministic case, if we have $A(y, \vec{x}) \rhd_d A(y+1, \vec{x})$ it is enough to glue all copies of the sequences of reductions for $0 \leq y \leq s$, to have $A(0, \vec{x}) \rhd_d A(s, \vec{x})$. More precisely, assume that all reductions have the same length $t'(\vec{x})$ greater than $t(s, \vec{x})$. This is an immediate consequence of the facts that we can find a monotone majorization for $t(y, \vec{x})$ like $r(y, \vec{x})$, and since $y \leq s$ we have $t(y, \vec{x}) \leq r(y, \vec{x}) \leq r(s, \vec{x})$. Now it is enough to repeat the last formula in the flow to make the flow longer to reach the length $t'(\vec{x}, \vec{z})=r(s, \vec{x})$ where $\vec{z}$ is a vector of variables in $s$. Now, define $t''(\vec{x}, \vec{z})=s \times (t'(\vec{x})+2)$,
\[
I(u, \vec{x})=
\begin{cases}
H(u, y, \vec{x}) & y(t'+2) < u < (y+1)(t'+2)\\
A(y, \vec{x}) & u= y(t'+2)\\
\end{cases} 
\]
and
\[
F(u)=
\begin{cases}
F(u, y) & y(t'+2) < u < (y+1)(t'+2)-1\\
E_0(u, y) & u=y(t'+2)\\
G_1(u, y+1) & u= (y+1)(t'+2)-1\\
\end{cases} 
\]
and $E'_0=E'_1=G'_0=G'_1=id$. It is easy to see that this new sequence is a deterministic $(\Pi_k(\Phi), \mathcal{B})$-flow from $A(0, \vec{x})$ to $A(s, \vec{x})$.\\
For the non-deterministic case, notice that $\mathbb{T}$ is closed under substitution, sum and product and therefore, $t'' \in \mathbb{T}$ which makes $(t'', I)$ a non-deterministic $(\mathbb{T}, \Pi, \mathcal{B})$-flow from $A(0, \vec{x})$ to $A(s, \vec{x})$.
\end{proof}
\begin{lem}\label{t2-23}(Conjunction and Disjunction Rules)
\begin{itemize}
\item[$(i)$]
If $\Gamma, A \rhd \Delta$ or $\Gamma, B \rhd \Delta$ then $\Gamma, A \wedge B \rhd \Delta$.
\item[$(ii)$]
If $\Gamma_0 \rhd \Delta_0, A $ and $\Gamma_1 \rhd \Delta_1, B$ then $\Gamma_0, \Gamma_1 \rhd \Delta_0, \Delta_1, A \wedge B$.
\item[$(iii)$]
If $\Gamma \rhd \Delta, A$ or $\Gamma \rhd \Delta, B$ then $\Gamma \rhd \Delta, A \vee B$.
\item[$(iv)$]
If $\Gamma_0, A \rhd \Delta_0$ and $\Gamma_1, B \rhd \Delta_1$ then $\Gamma_0, \Gamma_1, A \vee B \rhd \Delta_0, \Delta_1$.
\end{itemize}
\end{lem}
\begin{proof}
$(i)$ and $(iii)$, for both deterministic and non-deterministic cases, are trivial. For $(ii)$, in the deterministic case, if $\Gamma_0 \rhd \Delta_0, A $, then by conjunction application with $\bigwedge \Gamma_1$ we have $\bigwedge \Gamma_0 \wedge \bigwedge \Gamma_1 \rhd (\bigvee \Delta_0 \vee A) \wedge \bigwedge \Gamma_1$. Moreover, we have $\bigwedge \Gamma_1 \rhd \bigvee \Delta_1 \vee B$ and again by conjunction application $\bigwedge \Gamma_1 \wedge (\bigvee \Delta_0 \vee A) \rhd (\bigvee \Delta_1 \vee B) \wedge (\bigvee \Delta_0 \vee A) $. Therefore by weak gluing
\[
\bigwedge \Gamma_0 \wedge \bigwedge \Gamma_1 \rhd (\bigvee \Delta_1 \vee B) \wedge (\bigvee \Delta_0 \vee A) .
\]
But it is easy to see that 
\[
(\bigvee \Delta_1 \vee B) \wedge (\bigvee \Delta_0 \vee A)  \leq_d \bigvee \Delta_1 \vee \bigvee \Delta_0 \vee (A \wedge B).
\]
Hence
\[
\Gamma_0, \Gamma_1 \rhd \Delta_0, \Delta_1, (A \wedge B).
\]

For $(iv)$, if $\Gamma_0, A \rhd \Delta_0 $ then by disjunction application with $\bigwedge \Gamma_1 \wedge B$ we have 
\[
(\bigwedge \Gamma_0 \wedge A) \vee (\bigwedge \Gamma_1 \wedge B) \rhd \bigvee \Delta_0 \vee (\bigwedge \Gamma_1 \wedge B).
\]
Moreover, we have $\bigwedge \Gamma_1 \wedge B \rhd \bigvee \Delta_1 $, hence again by disjunction application 
\[
(\bigwedge \Gamma_1 \wedge B) \vee \bigvee \Delta_0 \rhd \bigvee \Delta_0 \vee \bigvee \Delta_1.
\]
Hence, by weak gluing,
\[
(\bigwedge \Gamma_0 \wedge A) \vee (\bigwedge \Gamma_1 \wedge B) \rhd \bigvee \Delta_0 \vee \bigvee \Delta_1.
\]
However, it is clear that 
\[
\bigwedge \Gamma_0 \wedge \bigwedge \Gamma_1 \wedge (A \vee B) \leq_d (\bigwedge \Gamma_0 \wedge A) \vee (\bigwedge \Gamma_1 \wedge B) .
\]
Hence,
\[
\Gamma_0, \Gamma_1, (A \vee B) \rhd \Delta_0, \Delta_1.
\]
\end{proof}
The following lemma makes it possible to compute a characteristic function of any $A \in \Psi_k \in \{\Pi_k(\Phi), \Sigma_k(\Phi)\}$ by a $\Sigma_{k+1}(\Phi)$ deterministic flow of reductions. This is a very important tool to reduce the complexity of deciding a complex formula to just deciding one equality. We will see its use in full force in the case of handling the contraction rule.
\begin{lem}\label{t2-25}(Computability of Characteristic Functions) \\Suppose $\{\Sigma_k(\Phi), \Pi_k(\Phi)\}_{k=0}^{\infty}$ is a hierarchy and $\mathcal{B}$ has characteristic terms for all $\phi \in \Phi$, then for any $\Psi \in \{\Pi_k(\Phi), \Sigma_k(\Phi)\}$ if $A(\vec{x}) \in \Psi$ then
\[
\rhd_d^{(\Sigma_{k+1}, \mathcal{B})} \; \exists i \leq 1 \; [(i=0 \rightarrow A) \wedge (i=1 \rightarrow \neg A)].
\]
\end{lem}
\begin{proof}
We prove the theorem by using induction on the number of bounded prefix quantifiers of $A$. If $A \in \Pi_0(\Phi)$, then there is nothing to prove because it is enough to put $i=\chi_A$ which belongs to the terms of $\mathcal{B}$ by the assumption. If $A=\forall z \leq t(\vec{x}) B(z, \vec{x})$, then by IH we have 
\[
\rhd_d^{(\Sigma_{k+1}, \mathcal{B})} \; \exists r \leq 1 \; [(r=1 \rightarrow B(\vec{x}, u+1)) \wedge (r=0 \rightarrow \neg B(\vec{x}, u+1))].
\]
Now, we want to prove that there exists a reduction from the conjunction of
\[
G(u+1)=\exists k \leq 1 \; [(k=1 \rightarrow B(\vec{x}, u+1)) \wedge (k=0 \rightarrow \neg B(\vec{x}, u+1))] 
\]
and
\[
H(u)=\exists i \leq 1 \; [(i=1 \rightarrow \forall z \leq u \; B(\vec{x}, z)) \wedge (i=0 \rightarrow \exists z \leq u \; \neg B(\vec{x}, z) )] 
\]
to
\[
H(u+1)=\exists j \leq 1 \; [(j=1 \rightarrow \forall z \leq u+1 \; B(\vec{x}, z)) \wedge (j=0 \rightarrow \exists z \leq u+1 \; \neg B(\vec{x}, z) )] .
\] 
Witness $j$ as the following:
\[
j=
\begin{cases}
1 & i=k=1 \\
0 & o.w. \\
\end{cases} 
\]
Then for the other quantifiers use the following scheme: If $i=k=1$, then witness $\forall z \leq u+1 \; B(\vec{x}, z)$ by $\forall z \leq u \; B(\vec{x}, z)$ and $B(\vec{x}, u+1)$. If $i=1$ and $k=0$, then to witness $\exists z \leq u+1 \; \neg B(\vec{x}, z)$ use $\neg B(\vec{x}, u+1)$ and finally if $i=k=0$, then use $\exists z \leq u \; \neg B(\vec{x}, z)$ to witness $\exists z \leq u+1 \; \neg B(\vec{x}, z)$.\\
Therefore $G(u+1) \wedge H(u) \rhd_d H(u+1)$. By IH, $\rhd_d G(u+1)$. Hence, by conjunction application $H(u) \rhd_d G(u+1) \wedge H(u)$ and then by gluing $H(u) \rhd_d H(u+1)$ and finally by strong gluing $H(0) \rhd_d H(t(\vec{x}))$. Since $H(0) \equiv_d G(0)$ and $\rhd_d G(0)$, hence $\rhd_d H(0)$ which means $\rhd_d H(t(\vec{x}))$. 
\end{proof}
\begin{lem}\label{t2-26}(Negation Rules) If $\Gamma, \Delta \subseteq \Pi_{k+1}$ and $A \in \Pi_k \cup \Sigma_k$ then
\begin{itemize}
\item[$(i)$]
If $\Gamma, A \rhd \Delta$ then $\Gamma \rhd \Delta, \neg A$.
\item[$(ii)$]
If $\Gamma \rhd \Delta, A$ then $\Gamma, \neg A \rhd \Delta$.
\end{itemize}
\end{lem}
\begin{proof}
Since we have conjunction and disjunction application, it is enough to prove that 
\begin{itemize}
\item[$(i)$]
$ \top \rhd^{\Pi_{k+1}} A(\vec{x}) \vee \neg A(\vec{x}) $.
\item[$(ii)$]
$A(\vec{x}) \wedge \neg A(\vec{x}) \rhd^{\Pi_{k+1}} \bot$.
\end{itemize}
The reason for this sufficiency is the following:\\

For $(i)$, if we have $\Gamma, A \rhd \Delta$ then $\bigwedge \Gamma \wedge A \rhd \bigvee \Delta$, hence by disjunction application we have 
$(\bigwedge \Gamma \wedge A) \vee \neg A \rhd \bigvee \Delta \vee \neg A$. By the claim we have $\rhd A \vee \neg A$, therefore by conjunction application $\bigwedge \Gamma \rhd \bigwedge \Gamma \wedge (A \vee \neg A)$. But, it is easy to see that $\Gamma \wedge (A \vee \neg A) \rhd (\bigwedge \Gamma \wedge A) \vee \neg A$. Hence by gluing we have $\bigwedge \Gamma \rhd \bigvee \Delta \vee \neg A$.\\

For $(ii)$, we have $\bigwedge \Gamma \rhd \bigvee \Delta \vee A$. By conjunction application $\bigwedge \Gamma \wedge \neg A \rhd (\bigvee \Delta \vee A) \wedge \neg A$. By the claim we have $ A \wedge \neg A \rhd \bot$ therefore by disjunction application $\bigvee \Delta \vee (A \wedge \neg A) \rhd \bigvee \Delta $. But, it is clear that $(\bigvee \Delta \vee A) \wedge \neg A \rhd \bigvee \Delta \vee (A \wedge \neg A)$. Hence by gluing, $\bigwedge \Gamma \wedge \neg A \rhd \bigvee \Delta$.\\

Now, we will prove the claim. For the non-deterministic case, the claim is trivial because we have $\mathcal{B} \vdash \top \rightarrow A(\vec{x}) \vee \neg A(\vec{x}) $ and $\mathcal{B} \vdash A(\vec{x}) \wedge \neg A(\vec{x}) \rightarrow \bot$.\\

For the deterministic case we will prove the existence of a $\Sigma_{k+1}$-flow. Then the claim will be clear using negation on all the elements of the flow. For $(i)$, notice that 
\[
\exists i \leq 1 \; [(i=0 \rightarrow A) \wedge (i=1 \rightarrow \neg A)] \leq_d A \vee \neg A.
\]
It is enough to witness $A$ and $\neg A$ in both sides with themselves. But since 
\[
\rhd_d^{(\Sigma_{k+1}, \mathcal{B})} \; \exists i \leq 1 \; [(i=0 \rightarrow A) \wedge (i=1 \rightarrow \neg A)],
\]
we can deduce $\rhd_d^{(\Sigma_{k+1}, \mathcal{B})} A \vee \neg A$.
For the deterministic case of $(ii)$, notice that $A \wedge \neg A \leq_d \bot$ because it is enough to witness the quantifiers of $A$ by $\neg A$ and vice versa.
\end{proof}
As we observed in the Example \ref{t2-16}, the main difference between the deterministic and non-deterministic reductions is the contraction rule that the non-deterministic reduction can handle by definition and the deterministic reduction obviously can not. In the following lemma, we will show that it is possible to simulate the contraction rule by deterministic reductions in the cost of extending one reduction to a sequence of them, i.e., a flow. 
\begin{lem}\label{t2-27}(Structural rules) 
\begin{itemize}
\item[$(i)$]
If $\Gamma, A, B, \Sigma \rhd \Delta$ then $\Gamma, B, A, \Sigma \rhd \Delta$.
\item[$(ii)$]
If $\Gamma \rhd \Delta, A, B, \Sigma$ then $\Gamma \rhd \Delta, A, B, \Sigma$.
\item[$(iv)$]
If $\Gamma \rhd \Delta$ then $\Gamma, A \rhd \Delta$.
\item[$(v)$]
If $\Gamma \rhd \Delta$ then $\Gamma \rhd \Delta, A$.
\item[$(iii)$]
If $\Gamma, A, A \rhd \Delta$ then $\Gamma, A \rhd \Delta$.
\item[$(vi)$]
If $\Gamma \rhd \Delta, A, A$ then $\Gamma \rhd \Delta, A$.
\end{itemize}
\end{lem}
\begin{proof}
The weakening and the exchange cases are trivial. For the contraction case notice that since we have conjunction and disjunction applications and also the gluing rule, it is enough to prove the following claim:\\

\textbf{Claim.} If $\Psi \in \{\Sigma_k(\Phi), \Pi_k(\Phi)\}$ and $A \in \Psi$, then:
\begin{itemize}
\item[$(i)$]
$A(\vec{x}) \rhd^{\Psi} A(\vec{x}) \wedge A(\vec{x}) $.
\item[$(ii)$]
$A(\vec{x}) \vee A(\vec{x}) \rhd^{\Psi} A(\vec{x}) $.
\end{itemize}

For the non-deterministic case there is nothing to prove because the left side and the right side are provably equivalent. For the deterministic case of $(ii)$, use induction on the complexity of $A$. If $A \in \Phi$, then there is nothing to prove. If $A=\forall \vec{z} \leq \vec{t}(\vec{x}) \; B(\vec{x}, \vec{z})$, then since $B \in \Sigma_k(\Phi)$, by IH we will have $B(\vec{x}, \vec{z}) \vee B(\vec{x}, \vec{z}) \rhd^{\Sigma_k} B(\vec{x}, \vec{z}) $. Therefore, it is clear that 
\[
\forall \vec{u} \leq \vec{t}(\vec{x}) \; B(\vec{x}, \vec{u}) \vee \forall \vec{v} \leq \vec{t}(\vec{x}) B(\vec{x}, \vec{v}) \rhd^{\Pi_{k+1}} \forall \vec{z} \leq \vec{t}(\vec{x})\;  B(\vec{x}, \vec{z}),
\]
because it is enough to add $\forall \vec{z} \leq \vec{t}(\vec{x})$ in front of all formulas in the flow and then witness them by themselves. Hence,
\[
\forall \vec{z} \leq \vec{t}(\vec{x}) \; [B(\vec{x}, \vec{z}) \vee B(\vec{x}, \vec{z})] \rhd^{\Pi_k} \forall \vec{z} \leq \vec{t}(\vec{x})\;  B(\vec{x}, \vec{z}).
\]
And then we have to add
\[
\forall \vec{u} \leq \vec{t}(\vec{x}) \; B(\vec{x}, \vec{u}) \vee \forall \vec{v} \leq \vec{t}(\vec{x}) B(\vec{x}, \vec{v}) 
\]
as the first formula in the flow to have
\[
\forall \vec{u} \leq \vec{t}(\vec{x}) \; B(\vec{x}, \vec{u}) \vee \forall \vec{v} \leq \vec{t}(\vec{x}) B(\vec{x}, \vec{v}) \rhd^{\Pi_{k+1}} \forall \vec{z} \leq \vec{t}(\vec{x})\;  B(\vec{x}, \vec{z}).
\]
Notice that we have to use the variable $z$ as the witness for both of $u$ and $v$.\\

If $A=\exists \vec{z} \leq \vec{t}(\vec{x}) \; B(\vec{x}, \vec{z})$, then note that we have
$B(\vec{u}) \wedge \neg B(\vec{u}) \rhd_d \bot$ and $B(\vec{v}) \wedge \neg B(\vec{v}) \rhd_d \bot$ and hence by propositional rules 
\[
(B(\vec{u}) \vee B(\vec{v})) \wedge \neg B(\vec{u}) \wedge \neg B(\vec{v}) \rhd_d \bot \;\; (*)
\]
Therefore, there is a flow from
\[
[B(\vec{u}) \vee B(\vec{v})] \wedge \exists i,j \leq 1 \; (\chi_B(\vec{u})=i) \wedge (\chi_B(\vec{v})=j) 
\]
to
\[
\exists i,j \leq 1 \; [(\chi_B(\vec{u})=i) \wedge (\chi_B(\vec{v})=j)] \wedge (i=1 \vee j=1)
\]
where $\chi_B(\vec{u})=i$ means $(i=1 \rightarrow B(\vec{u})) \wedge (i=0 \rightarrow \neg B(\vec{u}))$.
It is enough to define the sequence of statements in between by the following scheme: If $i=j=1$, then use $B(u) \wedge B(v)$. If $i=1$ and $j=0$ use $B(u) \wedge \neg B(v)$. If $i=0$ and $j=1$ use $\neg B(u) \wedge B(v)$. And finally if $i=j=0$, use the flow from $(*)$.\\

Therefore, using the Lemma \ref{t2-25}, we know that there is a flow from 
\[
\exists \vec{u}, \vec{v} \leq \vec{t} \; B(\vec{u}) \vee B(\vec{v})
\]
to 
\[
\exists \vec{u}, \vec{v} \leq \vec{t} \; \exists i,j \leq 1 \; [(\chi_B(u)=i) \wedge (\chi_B(v)=j)] \wedge (i=1 \vee j=1).
\]
Now, it is enough to show that
\[
\exists \vec{u}, \vec{v}  \leq \vec{t}(\vec{x}) \; \exists i,j \leq 1 \; (i=1 \vee j=1) \wedge (\chi_B(u)=i) \wedge (\chi_B(v)=j)
\]
is reducible to $\exists \vec{y} \leq \vec{t}(\vec{x}) \; B(\vec{x}, \vec{y})$. It is enough to read $i$ and $j$ and decide between the cases that $i=1$ or $j=1$. Then based on that decision we can decide to witness $\exists \vec{y} \leq \vec{t}(\vec{x}) \; B(\vec{x}, \vec{y})$ as $\exists \vec{u} \leq \vec{t}(\vec{x}) \; B(\vec{x}, \vec{u})$ or $\exists \vec{v} \leq \vec{t}(\vec{x}) \; B(\vec{x}, \vec{v})$.\\

The case $(i)$ is the dual of $(ii)$ and provable by just taking negations.
\end{proof}

\begin{lem}\label{t2-28}(Cut and Induction)
\begin{itemize}
\item[$(i)$]
If $\Gamma_0(\vec{x}) \rhd A(\vec{x}), \Delta_0(\vec{x}) $ and $\Gamma_1(\vec{x}), A(\vec{x}) \rhd \Delta_1(\vec{x})$, then $\Gamma_0(\vec{x}), \Gamma_1(\vec{x}) \rhd \Delta_0(\vec{x}), \Delta_1(\vec{x})$.
\item[$(ii)$]
If $s \in \mathbb{T}$ and $\Gamma(\vec{x}), A(y, \vec{x}) \rhd \Delta(\vec{x}), A(y+1, \vec{x})$, then $\Gamma(\vec{x}), A(0, \vec{x}) \rhd \Delta(\vec{x}), A(s(\vec{z}, \vec{x}), \vec{x})$.
\end{itemize}
\end{lem}
\begin{proof}
For $(i)$, Since $\Gamma_0 \rhd \Delta_0, A$ and $ \Gamma_1, A \rhd \Delta_1$ then $\bigwedge \Gamma_0 \rhd \bigvee\Delta_0 \vee A$ and $\bigwedge \Gamma_1 \wedge A \rhd \bigvee \Delta_1$. Apply conjunction with $\bigwedge \Gamma_1$ on the first one and disjunction with $\bigvee \Delta_0$ on the second one to prove $\bigwedge \Gamma_1 \wedge \bigwedge \Gamma_0 \rhd (\bigvee\Delta_0 \vee A) \wedge \bigwedge \Gamma_1$ and $(\bigwedge \Gamma_1 \wedge A) \vee \bigvee \Delta_0  \rhd \bigvee \Delta_1 \vee \bigvee \Delta_0$. Since $(\bigvee\Delta_0 \vee A) \wedge \bigwedge \Gamma_1 \leq (\bigwedge \Gamma_1 \wedge A) \vee \bigvee \Delta_0 $, by using gluing we will have $\bigwedge \Gamma_1 \wedge \bigwedge \Gamma_0 \rhd \bigvee \Delta_0 \vee \bigvee \Delta_1$.\\

For $(ii)$ we reduce the induction case to the strong gluing case. Since $\Gamma, A(y, \vec{x}) \rhd \Delta, A(y+1, \vec{x})$ by definition, $\bigwedge \Gamma \wedge A(y, \vec{x}) \rhd \bigvee \Delta \vee A(y+1, \vec{x})$. Therefore, by the Lemma \ref{t2-22} we have
\[
(\bigwedge \Gamma \wedge A(y, \vec{x})) \vee \bigvee \Delta \rhd  \bigvee \Delta \vee A(y+1, \vec{x}) \vee \bigvee \Delta
\]
and by contraction for $\bigvee \Delta$ we know 
\[
\bigvee \Delta \vee A(y+1, \vec{x}) \vee \bigvee \Delta \rhd \bigvee \Delta \vee A(y+1, \vec{x}).
\]
Hence,
\[
(\bigwedge \Gamma \wedge A(y, \vec{x})) \vee \bigvee \Delta \rhd  \bigvee \Delta \vee A(y+1, \vec{x}).
\]
Then by conjunction introduction and the fact that $(\bigwedge \Gamma \wedge A(y, \vec{x})) \vee \bigvee \Delta) \rhd \bigwedge \Gamma \vee \bigvee \Delta$,
\[
((\bigwedge \Gamma \wedge A(y, \vec{x})) \vee \bigvee \Delta), (\bigwedge \Gamma \wedge A(y, \vec{x})) \vee \bigvee \Delta) \rhd (\bigvee \Delta \vee A(y+1, \vec{x})) \wedge (\bigwedge \Gamma \vee \bigvee \Delta)
\]
By using the propositional, structural and the cut rule, it is easy to prove
\[
(\phi \vee \psi) \wedge (\sigma \vee \psi) \rhd (\phi \wedge \sigma) \vee \psi.
\]
Hence, by using the contraction we have
\[
(\bigwedge \Gamma \wedge A(y, \vec{x})) \vee \bigvee \Delta \rhd (\bigwedge \Gamma \wedge A(y+1, \vec{x})) \vee \bigvee \Delta.
\]
Now by strong gluing we have
\[
(\bigwedge \Gamma \wedge A(0, \vec{x})) \vee \bigvee \Delta \rhd (\bigwedge \Gamma \wedge A(s(\vec{z}, \vec{x}), \vec{x})) \vee \bigvee \Delta.
\]
But since $\Gamma \wedge A(0, \vec{x}) \rhd (\bigwedge \Gamma \wedge A(0, \vec{x})) \vee \bigvee \Delta $
and
\[
(\bigwedge \Gamma \wedge A(s(\vec{x}), \vec{x})) \vee \bigvee \Delta  \leq \bigvee \Delta \vee A(s(\vec{z}, \vec{x}), \vec{x}),
\]
we have
\[
\Gamma(\vec{x}), A(0, \vec{x}) \rhd \Delta(\vec{x}), A(s(\vec{z}, \vec{x}), \vec{x}).
\]
\end{proof}

\begin{lem}\label{t2-29}(Implication Rules) 
\begin{itemize}
\item[$(i)$]
If $\Gamma_0 \rhd \Delta_0, A $ and $\Gamma_1, B \rhd \Delta_1$ then $\Gamma_0, \Gamma_1, A \rightarrow B \rhd \Delta_0, \Delta_1$.
\item[$(ii)$]
If $\Gamma, A \rhd \Delta, B$ then $\Gamma \rhd \Delta, A \rightarrow B$.
\end{itemize}
In the deterministic case, we assume $A \rightarrow B \in \Pi_k(\Phi)$.
\end{lem}
\begin{proof}
For $(i)$, in the deterministic case notice that if $A \rightarrow B \in \Pi_k(\Phi)$ then $k=0$ and since $\Phi$ is closed under subformulas, $A, B \in \Phi$ and $\neg A \in \Phi$. Therefore, by definition, it is easy to see that $A \rightarrow B \equiv \neg A \vee B$. Therefore:\\

For $(i)$ since $\Gamma_0 \rhd \Delta_0, A $ by the Lemma \ref{t2-26} we have $\Gamma_0, \neg A \rhd \Delta_0$. On the other hand, we have $\Gamma_1, B \rhd \Delta_1$. Therefore, by the Lemma \ref{t2-23} we have $\Gamma_0, \Gamma_1, \neg A \vee B \rhd \Delta_0, \Delta_1$. Since $A \rightarrow B \rhd \neg A \vee B$, by using cut we have$\Gamma_0, \Gamma_1, A \rightarrow B \rhd \Delta_0, \Delta_1$. \\

For $(ii)$, if we have $\Gamma, A \rhd \Delta, B$ then by the Lemma \ref{t2-26} we have $\Gamma, \rhd \Delta, \neg A, B$. Hence by the Lemma \ref{t2-23} we have $\Gamma, \rhd \Delta, (\neg A \vee B), (\neg A \vee B)$. By contraction, $\Gamma, \rhd \Delta, (\neg A \vee B)$. Since $\neg A \vee B \rhd A \rightarrow B$, by cut $\Gamma, \rhd \Delta, A \rightarrow B$.\\

For the non-deterministic case note that when $A \rightarrow B \in \Pi$ then since $\Pi$ is closed under subformulas, we have $A, B \in \Pi$. For $(i)$, since $\Gamma_0 \rhd \Delta_0, A $ by conjunction application we have 
\[
\bigwedge \Gamma_0 \wedge A \rightarrow B \rhd (\bigvee \Delta_0 \vee A) \wedge A \rightarrow B.
\]
Since 
\[
(\bigvee \Delta_0 \vee A) \wedge (A \rightarrow B) \rhd \bigvee \Delta_0 \vee (A \wedge (A \rightarrow B)),
\]
and $A \wedge A \rightarrow B \leq_n B$, we have 
\[
\bigvee \Delta_0 \vee (A \wedge (A \rightarrow B)) \rhd \bigvee \Delta_0 \vee B.
\]
And then since $\Gamma_1 \rhd B, \Delta_1$, by cut on $B$ we have
\[
\Gamma_0, \Gamma_1, A \rightarrow B \rhd \Delta_0, \Delta_1.
\]
\\
For $(ii)$, if $\Gamma, A \rhd B, \Delta$, then by disjunction application 
\[
(\bigwedge \Gamma \wedge A) \vee (A \rightarrow B) \rhd \bigvee \Delta \vee B \vee (A \rightarrow B).
\]
And since
\[
((\bigwedge \Gamma \vee (A \rightarrow B)) \wedge (A \vee (A \rightarrow B)) \rhd (\bigwedge \Gamma \wedge A) \vee (A \rightarrow B),
\]
we have 
\[
((\bigwedge \Gamma \vee (A \rightarrow B)) \wedge (A \vee (A \rightarrow B)) \rhd \bigvee \Delta \vee B \vee (A \rightarrow B).
\]
Since $B \leq_n (A \rightarrow B)$, by contraction and cut we have $B \vee (A \rightarrow B) \rhd A \rightarrow B$. On the other hand, $ \leq A \vee (A \rightarrow B)$. Hence 
\[
\Gamma \rhd ((\bigwedge \Gamma \vee (A \rightarrow B)) \wedge (A \vee (A \rightarrow B)),
\]
and therefore by gluing $\Gamma \rhd \Delta, A \rightarrow B$.
\end{proof}
The following theorem is the main theorem of the theory of flows in bounded theories of arithmetic:
\begin{thm}\label{t2-30}(Soundness)
\begin{itemize}
\item[$(i)$]
If $\Gamma(\vec{x}) \cup \Delta(\vec{x}) \subseteq \Pi_k(\Phi)$, $\mathfrak{B}(\mathbb{T}_{all}, \Pi_k(\Phi), \mathcal{A}) \vdash \Gamma(\vec{x}) \Rightarrow \Delta(\vec{x})$ and $ \mathcal{A} \subseteq \mathcal{B}$ then $ \Gamma \rhd_{d}^{(\Pi_k(\Phi), \mathcal{B})} \Delta$.
\item[$(ii)$]
If $\Pi$ is a $\pi$-class, $\Gamma(\vec{x}) \cup \Delta(\vec{x}) \subseteq \Pi$, $\mathfrak{B}(\mathbb{T}, \Pi, \mathcal{A}) \vdash \Gamma(\vec{x}) \Rightarrow \Delta(\vec{x})$ and $ \mathcal{A} \subseteq \mathcal{B}$ then $ \Gamma \rhd_{n}^{(\mathbb{T}, \Pi, \mathcal{B})} \Delta$.
\end{itemize}
\end{thm}
\begin{proof}
We prove the lemma by induction on the length of the free-cut free proof of $\Gamma(\vec{x}) \Rightarrow \Delta(\vec{x})$.\\

1. (Axioms). If $\Gamma(\vec{x}) \Rightarrow \Delta(\vec{x})$ is a logical axiom then the claim is trivial. If it is a non-logical axiom then the claim will be also trivial because all non-logical axioms are quantifier-free and provable in $\mathcal{B}$. Therefore there is nothing to prove.\\

2. (Structural Rules). It is proved in the Lemma \ref{t2-27}.\\ 

3. (Cut). It is proved by Lemma \ref{t2-28}.\\

4. (Propositional). The conjunction and disjunction cases are proved in the Lemma \ref{t2-23}. The implication and negation cases are proved in the Lemma \ref{t2-29}. \\

5. (Bounded Universal Quantifier, Right). If $\Gamma(\vec{x}) \Rightarrow \Delta(\vec{x}), \forall z \leq p(\vec{x})  B(\vec{x}, z)$ is proved by the $\forall^{\leq} R$ rule by $\Gamma(\vec{x}), z \leq p(\vec{x}) \Rightarrow \Delta(\vec{x}), B(\vec{x}, z)$, then by IH $\Gamma(\vec{x}), z \leq p(\vec{x}) \rhd_{d} \Delta(\vec{x}), B(\vec{x}, z)$. Therefore, there exists a term $t(\vec{x})$, a formula $H(u, \vec{x}, z) \in \Pi_k(\Phi)$ and sequences of terms $E_0$ $E_1$, $G_0$, $G_1$ and $F(u)$ such that the conditions of the Definition \ref{t2-12} are provable in $\mathcal{B}$. First of all, extend the sequence by repeating the last formula to reach a majorization $t'(\vec{x})$ of $t(\vec{x}, p(\vec{x}))$. This is possible because $z \leq p(\vec{x})$ and $t$ is monotone. Then, define $t'(\vec{x}) \geq t(\vec{x}, p(\vec{x}))$ and $H'(u, \vec{x})=\forall z \leq p(\vec{x}) H(u, \vec{x}, z)$ and finally define $E'_0$ $E'_1$, $G'_0$, $G'_1$ and $F'(u)$ as functions that read the outmost quantifier $\forall z$ and sends it to itself and then apply the corresponding operations. Since $H(u, \vec{x}, z) \in \Pi_k(\Phi)$, then $\forall z \leq p(\vec{x}) H(u, \vec{x}, z) \in \Pi_k(\Phi)$. The other conditions to check that the new sequence is a $(\Pi_k(\Phi), \mathcal{B})$-flow is straightforward .\\ 

5$'$. For the non-deterministic case, by IH we have $\Gamma(\vec{x}), z \leq p(\vec{x}) \rhd_{n} \Delta(\vec{x}), B(\vec{x}, z)$. Therefore, there exists a term $t(\vec{x}) \in \mathbb{T}$, a formula $H(u, \vec{x}, z) \in \Pi$ such that the conditions of the Definition \ref{t2-13} are provable in $\mathcal{B}$. First of all, extend the sequence by repeating the last formula to reach a majorization $t'(\vec{x})$ of $t(\vec{x}, p(\vec{x}))$. This is possible since $z \leq p(\vec{x})$ and $t$ is monotone. Then, define $t'(\vec{x}) \geq t(\vec{x}, p(\vec{x}))$ and $H'(u, \vec{x})=\forall z \leq p(\vec{x}) H(u, \vec{x}, z)$. Since $H(u, \vec{x}, z) \in \Pi$ then $\forall z \leq p(\vec{x}) H(u, \vec{x}, z) \in \Pi$. The other conditions to check that the new sequence is a $(\mathbb{T}, \Pi, \mathcal{B})$-flow is a straightforward consequence of the fact that if $\mathcal{B} \vdash \forall u \leq t'(\vec{x}) H(u, z, \vec{x}) \rightarrow H(u+1, z, \vec{x})$, then
\[
\mathcal{B} \vdash \forall u \leq t'(\vec{x}) \forall z \leq p(\vec{x}) H(u, z, \vec{x}) \rightarrow \forall z \leq p(\vec{x}) H(u+1, z, \vec{x}).
\]
\\
6. (Bounded Universal Quantifier, Left). Suppose $\Gamma(\vec{x}), s(\vec{x}) \leq p(\vec{x}), \forall z \leq p(\vec{x}) B(\vec{x}, z) \Rightarrow \Delta(\vec{x})$ is proved by the $\forall^{\leq} L$ rule by $\Gamma(\vec{x}), B(\vec{x}, s(\vec{x})) \Rightarrow \Delta(\vec{x})$. Then by IH, $\Gamma(\vec{x}), B(\vec{x}, s(\vec{x})) \rhd_{d} \Delta(\vec{x})$. Therefore, there exist a term $t(\vec{x})$, a formula $H(u, \vec{x}) \in \Pi_k(\Phi)$ and sequences of terms $E_0$ $E_1$, $G_0$, $G_1$ and $F(u)$ such that the conditions of the Definition \ref{t2-12} are provable in $\mathcal{B}$. Similar to the case 5, w.l.o.g. extend the length to $t'(\vec{x}) \geq t(\vec{x}, p(\vec{x}))$. Now define $t''(\vec{x})=t'(\vec{x})+1$, 
\[
H'(u, \vec{x})=
\begin{cases}
\bigwedge \Gamma(\vec{x}) \wedge s(\vec{x}) \leq p(\vec{x}) \wedge \forall z \leq p(\vec{x}) B(\vec{x}, z) & u=0\\
H(u, y, \vec{x}) & 0<u \leq t'(\vec{x})+1\\
\end{cases} 
\]
And finally, define $E'_0$ $E'_1$, $G'_0$, $G'_1$ and $F'(u)$ as sequences of terms that compute the universal quantifier $\forall z$ as $t(\vec{x})$. Since $t(\vec{x})$ is a terms, it is easy to check that this new sequence is the $(\Pi_k(\Phi), \mathcal{B})$-flow that we wanted.\\

6$'$. For the non-deterministic case, since $\mathcal{B} \vdash s(\vec{x}) \leq p(\vec{x}) \wedge \forall z \leq p(\vec{x}) B(\vec{x}, z) \rightarrow B(\vec{x}, s(\vec{x}))$, we have
\[
s(\vec{x}) \leq p(\vec{x}), \forall z \leq p(\vec{x}) B(\vec{x}, z) \rhd_n B(\vec{x}, s(\vec{x})).
\]
Since
\[
\Gamma(\vec{x}), B(\vec{x}, s(\vec{x})) \rhd_{n} \Delta(\vec{x}),
\]
by cut we have
\[
\Gamma(\vec{x}), s(\vec{x}) \leq p(\vec{x}), \forall z \leq p(\vec{x}) B(\vec{x}, s(\vec{x})) \rhd_{d} \Delta(\vec{x}).
\]
\\
7. (Bounded Existential Quantifier, Right). If $\Gamma(\vec{x}), s(\vec{x}) \leq p(\vec{x}) \Rightarrow \Delta(\vec{x}), \exists z \leq p(\vec{x}) B(\vec{x}, z)$ is proved by the $\exists^{\leq} R$ rule by $\Gamma(\vec{x}) \Rightarrow \Delta(\vec{x}), B(\vec{x}, s(\vec{x}))$ then $\neg B(\vec{x}, z) \in \Pi$. Therefore, by Lemma \ref{t2-26} $\Gamma(\vec{x}), \neg B(\vec{x}, s(\vec{x})) \Rightarrow \Delta(\vec{x})$. By 6, $\Gamma(\vec{x}), s(\vec{x}) \leq p(\vec{x}), \forall z \leq p(\vec{x}) \neg B(\vec{x}, z) \Rightarrow \Delta(\vec{x})$. Again by Lemma \ref{t2-26}, $\Gamma(\vec{x}), s(\vec{x}) \leq p(\vec{x}) \Rightarrow \Delta(\vec{x}), \exists z \leq p(\vec{x}) \neg \neg B(\vec{x}, z)$ which means $\Gamma(\vec{x}), s(\vec{x}) \leq p(\vec{x}) \Rightarrow \Delta(\vec{x}), \exists z \leq p(\vec{x}) B(\vec{x}, z)$. \\

8. (Bounded Existential Quantifier, Left). If $\Gamma, y \leq p(\vec{x}), B(\vec{x}, y) \rhd \Delta$ then since $\exists y \leq p(\vec{x}) B(\vec{x}, y) \in \Pi$, then $B$ has a negation in $\Pi$. By disjunction application $(\Gamma \wedge B(\vec{x}, y)) \vee \neg B(\vec{x}, y) \rhd \bigvee \Delta \vee \neg  B(\vec{x}, y)$. Since $\rhd B(\vec{x}, y) \vee \neg B(\vec{x}, y) $, then $\Gamma \rhd (\Gamma \wedge B(\vec{x}, y)) \vee \neg B(\vec{x}, y)$. Therefore, $\Gamma, y \leq p(\vec{x}) \rhd \bigvee \Delta \vee \neg B(\vec{x}, y)$. Now by 5, we have 
\[
\Gamma \rhd \Delta, \forall y \leq p(\vec{x}) \neg B(\vec{x}, y).
\]
By conjunction application
\[
\Gamma, \exists y \leq p(\vec{x}) B(\vec{x}, y) \rhd (\bigvee \Delta \vee \forall y \leq p(\vec{x}) \neg B(\vec{x}, y)) \wedge \exists y \leq p(\vec{x}) B(\vec{x}, y).
\]
But,
\[
\forall y \leq p(\vec{x}) \neg B(\vec{x}, y) \wedge \exists y \leq p(\vec{x}) B(\vec{x}, y) \rhd \bot.
\]
Hence,
\[
(\bigvee \Delta \vee \forall y \leq p(\vec{x}) \neg B(\vec{x}, y)) \wedge \exists y \leq p(\vec{x}) B(\vec{x}, y) \rhd \Delta.
\]
Therefore,
\[
\Gamma, \exists y \leq p(\vec{x}) B(\vec{x}, y) \rhd \Delta.
\]
\\
9. (Induction). It is proved in Lemma \ref{t2-28}.
\end{proof}
Using the soundness theorem we can show that any non-deterministic reduction and hence all non-deterministic flows can be simulated by term-length deterministic flows. Note that even when the length of a non-deterministic flow belongs to some term ideal of the language, then the length of the simulated deterministic flow exceeds all the terms in the term ideal and needs the whole power of terms.
\begin{thm}\label{t2-31}(Simulation)
Let $\mathcal{B}$ be a bounded theory of arithmetic. Then all non-deterministic reductions can be simulated by a term-length sequence of deterministic reductions. In other words, if $A(\vec{x}) , B(\vec{x}) \in \Pi_k(\Phi)$ and $A(\vec{x}) \leq_{n}^{(\mathbb{T}, \Pi_k(\Phi), \mathcal{B})} B(\vec{x})$, then $A(\vec{x}) \rhd_d^{(\mathbb{T}_{all}, \Pi_k(\Phi), \mathcal{B})} B(\vec{x})$.
\end{thm}
\begin{proof}
If $A(\vec{x}) \leq_{n}^{\Pi} B(\vec{x})$ then $\mathcal{B} \vdash A(\vec{x}) \Rightarrow B(\vec{x})$. By deterministic soundness we have $A(\vec{x}) \rhd_d^{(\Pi_k(\Phi), \mathcal{B})} B(\vec{x})$.
\end{proof}

\begin{cor}\label{t2-32}
if $A(\vec{x}) , B(\vec{x}) \in \Pi_k(\Phi)$ and $A(\vec{x}) \rhd_{n}^{(\mathbb{T}, \Pi_k(\Phi), \mathcal{B})} B(\vec{x})$, then $A(\vec{x}) \rhd_d^{(\Pi_k(\Phi), \mathcal{B})} B(\vec{x})$. Therefore, the existence of a non-deterministic $(\mathbb{T}_{all}, \Pi_k(\Phi), \mathcal{B})$-flow is equivalent to the existence of a deterministic $(\Pi_k(\mathbb{T}_{all}), \mathcal{B})$-flow.
\end{cor}
\begin{proof}
If $A(\vec{x}) \rhd_{n}^{(\mathbb{T}, \Pi_k(\Phi), \mathcal{B})} B(\vec{x})$, then by definition, there exist a term $t(\vec{x})$ and a formula $H(u, \vec{x}) \in \Pi_k(\Phi)$ such that $A(\vec{x}) \equiv_n H(0, \vec{x})$, $B(\vec{x}) \equiv_n H(t(\vec{x}), \vec{x})$ and $H(u, \vec{x}) \leq_n H(u+1, \vec{x})$. By the Theorem \ref{t2-31}, $A(\vec{x}) \rhd_d H(0, \vec{x})$, $B(\vec{x}) \rhd_d H(t(\vec{x}), \vec{x})$ and $H(u, \vec{x}) \rhd_d H(u+1, \vec{x})$. By strong gluing, $H(0, \vec{x}) \rhd_d H(t(\vec{x}), \vec{x})$ and therefore by gluing $A(\vec{x}) \rhd_d B(\vec{x})$.
\end{proof}
We also have the following completeness theorem:
\begin{thm}\label{t2-33}(Completeness)
\begin{itemize}
\item[$(i)$]
If $\Gamma(\vec{x}) \rhd^{(\Pi_k(\Phi), \mathcal{B})}_{d} \Delta(\vec{x})$ and $\mathcal{B} \subseteq \mathfrak{B}(\mathbb{T}_{all}, \Pi_k(\Phi), \mathcal{A})$, then $\mathfrak{B}(\mathbb{T}_{all}, \Pi_k(\Phi), \mathcal{A}) \vdash \Gamma(\vec{x}) \Rightarrow \Delta(\vec{x})$.
\item[$(ii)$]
If $\Gamma(\vec{x}) \rhd^{(\mathbb{T}, \Pi, \mathcal{B})}_{n} \Delta(\vec{x})$ and $\mathcal{B} \subseteq \mathfrak{B}(\mathbb{T}, \Pi, \mathcal{A})$, then $\mathfrak{B}(\mathbb{T}, \Pi, \mathcal{A}) \vdash \Gamma(\vec{x}) \Rightarrow \Delta(\vec{x})$.
\end{itemize}
\end{thm}
\begin{proof}
For $(ii)$, if $\Gamma(\vec{x}) \rhd^{(\mathbb{T}, \Pi, \mathcal{B})}_{n} \Delta(\vec{x})$, then by Definition \ref{t2-13}, there exist a term $t(\vec{x}) \in \mathbb{T}$, and a formula $H(u, \vec{x}) \in \Pi$ such that we have the following:
\begin{itemize}
\item[$(i)$]
$\mathcal{B} \vdash H(0, \vec{x}) \leftrightarrow \bigwedge \Gamma(\vec{x})$,
\item[$(ii)$]
$\mathcal{B} \vdash H(t(x), \vec{x}) \leftrightarrow \bigvee \Delta(\vec{x})$,\\ and
\item[$(iii)$]
$\mathcal{B} \vdash H(u, \vec{x}) \rightarrow H(u+1, \vec{x}) $.
\end{itemize}
Since $\mathcal{B} \subseteq \mathfrak{B}(\mathbb{T}, \Pi, \mathcal{A})$, we have 
\[
\mathfrak{B}(\mathbb{T}, \Pi, \mathcal{A}) \vdash \forall u \leq t(\vec{x}) \; H(u, \vec{x}) \rightarrow H(u+1, \vec{x}).
\]
Since $H(u, \vec{x}) \in \Pi$ and $t \in \mathbb{T}$, by induction we have ,
\[
\mathfrak{B}(\mathbb{T}, \Pi, \mathcal{A}) \vdash H(0, \vec{x}) \rightarrow H(t(\vec{x}), \vec{x}).
\]
On the other hand, we have $\mathcal{B} \vdash H(0, \vec{x}) \leftrightarrow \bigwedge \Gamma(\vec{x})$ and $\mathcal{B} \vdash H(t(\vec{x}), \vec{x}) \leftrightarrow \bigvee \Delta(\vec{x})$. Therefore, $ \mathfrak{B}(\mathbb{T}, \Pi, \mathcal{A}) \vdash \Gamma(\vec{x}) \Rightarrow \Delta(\vec{x})$.\\
The proof of the deterministic case, i.e., the case $(i)$, is very similar. 
\end{proof}

\section{Applications on Bounded Theories}
In this section we will use the soundness theorems that we proved in the previous section to extract the computational content of the low complexity statements of some concrete weak bounded theories such as Buss's hierarchy of bounded theories of arithmetic and some strong theories such as $I\Delta_0+\EXP$ and $\PRA$. For the beginning, let us focus on the deterministic soundness theorem.
The first application is on the fragments of the theory $I\Delta_0$ which are related to the linear time hierarchy:
\begin{cor}\label{t3-1}
Let $\Gamma(\vec{x}) \cup \Delta(\vec{x}) \subseteq  \hat{U}_k$. Then, $I\hat{U}_k \vdash \Gamma(\vec{x}) \Rightarrow \Delta(\vec{x})$ iff $\Gamma \rhd_d^{(\hat{U}_k, \mathcal{R})} \Delta$.
\end{cor}
The second application, and maybe the more important one, is the case of Buss's hierarchy of bounded arithmetic. 
\begin{cor}\label{t3-2}
Let $\Gamma(\vec{x}) \cup \Delta(\vec{x}) \subseteq  \hat{\Pi}_{k}^b(\#_n)$. Then, $T^k_n \vdash \Gamma(\vec{x}) \Rightarrow \Delta(\vec{x})$ iff $\Gamma \rhd_d^{(\hat{\Pi}_{k}^b(\#_n), \PV(\#_n))} \Delta$. Specifically, for $n=2$, $T^k_2 \vdash \Gamma(\vec{x}) \Rightarrow \Delta(\vec{x})$ iff $\Gamma \rhd_d^{(\hat{\Pi}_{k}^b, \PV)} \Delta$.
\end{cor}
\begin{proof}
Note that it is enough to know that $T^k_n$ is axiomatizable by $\hat{\Pi}_{k}^b(\#_n)$-induction.
\end{proof}
And also we can apply the soundness theorem on stronger theories with full exponentiation like $I\Delta_0+\EXP$. Consider the theory $\mathcal{R}$ augmented with a function symbol for exponentiation with the usual recursive definition and denote it by $\mathcal{R}(exp)$. Then:  
\begin{cor}\label{t3-3}
Let $\Gamma(\vec{x}) \cup \Delta(\vec{x}) \subseteq  \Pi^b_k(open)$. Then, $I\Delta_0+\EXP \vdash \Gamma(\vec{x}) \Rightarrow \Delta(\vec{x})$ iff $\Gamma \rhd_d^{(\Pi^b_k(open), \mathcal{R}(exp))} \Delta$. 
\end{cor}
We can also use the theory of flows to extract the computational content of low complexity sentences of the very strong theories of arithmetic like $I\Sigma_n$ and $\PA+\TI(\alpha)$. But this is not what we can implement in a very direct way. The reason is that our method is tailored for bounded theories while these theories are unbounded. Hence, to use our theory, we have to find a way to transfer low complexity statements from these theories to some corresponding bounded theories. This is what the continuous cut elimination method makes possible in a very elegant way. It transfers all $\Pi^0_2$ consequences of a strong theory $T$ to some quantifier-free extensions of $\PRA$ and then makes it possible to apply the flow decomposition technique. To explain how it works, we need some definitions: 
\begin{dfn}\label{t3-4}
\begin{itemize}
\item[$(i)$]
An ordered structure $(X, \prec_X)$ is called $ \mathcal{B}$-representable when there exists a relation $\prec \; \in \mathcal{L}_{\mathcal{B}}$ defined by a quantifier-free formula such that:
\begin{itemize}
\item[$(i)$]
The order type of $\prec$ equals $\prec_X$.
\item[$(ii)$]
$\mathcal{B}$ proves the axioms of discrete ordered structures for the language.  
\end{itemize}
\item[$(ii)$]
An ordered structure $(X, \prec_X, +_X, \cdot_X, -_X, \lfloor \frac{\cdot}{\cdot} \rfloor_X, 0_X, 1_X) $ is called $\mathcal{B}$-representable when there exists a quantifier-free relation $\prec \; \in \mathcal{L}_{\mathcal{B}}$ and $\mathcal{L}_{\mathcal{B}}$-terms $+, \cdot, -, \lfloor \frac{\cdot}{\cdot} \rfloor : \mathbb{N} \times \mathbb{N} \to \mathbb{N}$ and constants $0, 1 \in \mathbb{N}$ such that:
\begin{itemize}
\item[$(i)$]
The order type of $\prec$ equals $\prec_X$.
\item[$(ii)$]
$\mathcal{B}$ proves the axioms of discrete ordered semi-rings for the language without the commutativity of addition and the axioms which state that $\prec$ preserves under left addition and left multiplication by a non-zero element.  
\end{itemize}
\end{itemize}
\end{dfn}
In this paper we are mainly interested in the cases that $\mathcal{B}=\PV$ or $\mathcal{B}=\PRA$, i.e., the case of polytime representability and the case of primitive recursive representability.
\begin{dfn}\label{t3-5}
Let $\prec$ be a quantifier-free formula in the language of $\PRA$. By theory $\PRA+\PRWO(\prec)$ we mean $\PRA$ plus the axiom schema $\PRWO(\prec)$ which states $\forall \vec{x} \exists y \; f(\vec{x}, y+1) \nprec f(\vec{x}, y) $ for any function symbol $f$.
\end{dfn}
The following theory is the skolemization of $\PRA+\PRWO(\prec)$:
\begin{dfn}\label{t3-6}
The language of the theory $\PRA_{\prec}$ consists of the language of $\PRA$ plus the scheme which says that for any $\PRA$-function symbol $f(\vec{x}, y)$, there exists a function symbol $[\mu y. f](\vec{x})$. Then $\BASIC_{\prec}$ is the theory axiomatized by the axioms of $\PRA$ and $\mathcal{R}$ and the following definitional equations: $f(\vec{x}, 1+[\mu y. f](\vec{x})) \nprec f(\vec{x}, \mu y. f](\vec{x}))$ and $z < [\mu y. f](\vec{x}) \rightarrow f(\vec{x}, z+1) \prec f(\vec{x}, z)$. Finally, $\PRA_{\prec}$ is $\BASIC_{\prec}$ plus the usual induction rule. 
\end{dfn}
We are ready to define $\Pi^0_2$-proof theoretical ordinal of a theory.
\begin{dfn}\label{t3-7}
Let $T$ be a theory of arithmetic. We say that $\alpha$ is a $\Pi^0_2$-proof theoretical ordinal of $T$ when $(\alpha, \prec_{\alpha})$ is $ \PRA$-representable by $\prec$ and $T \equiv_{\Pi^0_2} \PRA+\PRWO(\prec)$.
\end{dfn}
As we have mentioned before, using the continuous cut elimination technique, we can compute the $\Pi^0_2$-ordinal of some specific theories. (See \cite{Po} for the sketch of the proof for $\PA+\TI(\alpha)$. The rest is similar.)
\begin{thm}\label{t3-8}(Continuous Cut Elimination)
\begin{itemize}
\item[$(i)$]
The $\Pi^0_2$-ordinal of $I\Sigma_1$ is $\omega^2$.
\item[$(ii)$]
For $n > 1$, the $\Pi^0_2$-ordinal of $I\Sigma_n$ is $\omega_n$.
\item[$(iii)$]
The $\Pi^0_2$-ordinal of $\PA$ is $\epsilon_0$.
\item[$(iv)$]
For any $\PRA$-representable ordinal $\epsilon_0 \prec \alpha$, the $\Pi^0_2$-ordinal of $\PA+\TI(\alpha)$ is $\alpha$.
\end{itemize}

\end{thm}
Now we are ready to have the following corollary:
\begin{cor}\label{t3-9}
Let $\Gamma(\vec{x}) \cup \Delta(\vec{x}) \subseteq  \Pi^b_k(open)$, and $\alpha_T$ is the $\Pi^0_2$-ordinal of $T$ with a $\PRA$-representation $\prec_{\alpha_T}$, then $T \vdash \Gamma(\vec{x}) \Rightarrow \Delta(\vec{x})$ iff 
\[
\Gamma \rhd_d^{(\Pi^b_k(open), \BASIC_{\prec_{\alpha_T}})} \Delta.
\]
\end{cor}
\begin{proof}
Note that the existence of the flow is equivalent to the provability of $\Gamma \Rightarrow \Delta$ in $\PRA_{\prec_{\alpha_T}}$ because $\PRA_{\prec_{\alpha_T}}$ is a bounded theory axiomatizable by the usual induction on formulas in $\Pi^b_k(open)$. On the other hand, $\Gamma(\vec{x}) \cup \Delta(\vec{x}) \subseteq  \Pi^b_k(open)$, which means that the sequent is bounded and hence is in $\Pi^0_2$. Therefore, by the definition of $\Pi^0_2$-ordinal we know that $\PRA_{\prec_{\alpha_T}} \vdash \Gamma \Rightarrow \Delta$ iff $T \vdash \Gamma \Rightarrow \Delta$ and it completes the proof. 
\end{proof}
So far, we have used the theory of deterministic flows to decompose first order proofs of bounded theories. In the following we will introduce two different kinds of characterizations and we will use them to reprove some recent results for some specific classes of formulas. The types that we want to use are generalizations of some recent characterizations of some low complexity statements in Buss's hierarchy of bounded arithmetic by Game induction principles \cite{ST}, \cite{Ta} and some kind of PLS problems \cite{BB}.
\begin{dfn}\label{t3-10}
Fix a language $\mathcal{L}$. An instance of the $(j, k)$-game induction principle, $GI_k^j(\mathcal{L})$, is given by size parameters $a$ and $b$, a uniform sequence $G_0, \ldots, G_{a-1}$ of open (quantifier-free) relations, a term $V$ and a uniform sequence $W_0, \ldots, W_{a-2}$ of terms. The instance
$GI(G,V,W,a,b)$ states that, interpreting $G_0, \ldots, G_{a-1}$ as $k$-turn games in which all moves are bounded by $b$, the following cannot all
be true:
\begin{itemize}
\item[$(i)$]
Deciding the winner of game $G_0$ depends only on the first $j$ moves,
\item[$(ii)$]
Player $B$ can always win $G_0$ (expressed as a $\Pi_j(open)$ property.)
\item[$(iii)$]
For $i=0, \ldots, a-2$, $W_i$ gives a deterministic reduction of $G_{i+1}$ to $G_i$,
\item[$(iv)$]
$V$ is an explicit winning strategy for Player $A$ in $G_{a-1}$.
\end{itemize}
\end{dfn}
In the following theorem, denote $\Pi_i(\Phi)$ where $\Phi$ is the class of all quantifier-free formulas by $\Pi_i$ and do the similar thing for $\Sigma_i(\Phi)$. 
\begin{thm}\label{t3-11}
Let $j \leq k$. Then, 
\[
\forall \Sigma_j(open) (\mathfrak{B}(\mathbb{T}_{all}, \Pi_k(open), \mathcal{B})) \equiv^{\mathcal{B}} GI_k^j(\mathcal{L}).
\]
\end{thm}
\begin{proof}
It is clear that $\mathfrak{B}(\mathbb{T}_{all}, \Pi_k(open), \mathcal{B}) \vdash GI_k^j(\mathcal{L})$. For the converse, assume $\mathfrak{B}(\mathbb{T}_{all}, \Pi_k(open), \mathcal{B}) \vdash \forall x A(x)$ where $A \in \Sigma_j(open)$ and $j \leq k$. Then, we know that $\mathfrak{B}(\mathbb{T}_{all}, \Pi_k(open), \mathcal{B}) \vdash \neg A(x) \Rightarrow \bot$ and $\neg A \in \Pi_j(open)$. By Corollary \ref{t3-2}, there exist a term $t(x)$, a formula $H(u, x) \in \Pi_k(open)$ and sequences of terms $E_0$, $E_1$, $I_0$, $I_1$ and $F(u)$ such that the following statements are provable in $\mathcal{B}$:
\begin{itemize}
\item[$(i)$]
$H(0, x) \equiv_{d}^{(E_0, E_1)} \neg A(x)$.
\item[$(ii)$]
$H(t(x), x)\equiv_{d}^{(I_0, I_1)} \bot$.
\item[$(iii)$]
$\forall u < t(x) H(u, x) \leq_{d}^{F_u} H(u+1, x)$.
\end{itemize} 
First of all, note that we can change the definition of $H$ in the following way: 
\[
H'(u, x)=(u=0 \rightarrow \neg A(x) ) \wedge (u \neq 0 \rightarrow H(u-1, x)).
\]
And, it is possible to shift also the reductions to have $(i)$ to $(iii)$ for $H'$. But note that the truth of $H'(0, x)$ depends only on first $j$ blocks of quantifiers when we write it in the strict $\Pi_j(open)$ form.\\ 

W.l.o.g., we assume that all bounds in $H'(u, x)$ are the same, say $s(x)$. Since $H'$ is strict, we have $H'(u, x)=\forall \vec{z}_1 \leq s \exists \vec{y}_1 \leq s \forall \vec{z}_2 \leq s \ldots G(u, \vec{z}_1, \vec{y}_1, \vec{z}_2, \ldots)$. Define $a=t(x)$, $b=s(x)$, $G_i$ as the game $G(i, \vec{z}_1, \vec{y}_1, \vec{z}_2, \ldots)$, $W_i=F'_i$ and $V=I'_0$. Therefore, we have an instance of the game induction. Now we want to show that $A(x)$ is reducible to this game induction provably in $\mathcal{B}$. Since $\mathcal{B} \vdash \forall u < t(x) H'(u, x) \leq_{d}^{F'_u} H'(u+1, x)$ and $H(t(x), x)\equiv_{d}^{(I_0, I_1)} \bot$, the false part is ``player B can always win the game $G_0$" which means that $H'(0, x)$ is false. Since $H'(0, x)$ is equivalent with $A$ provable in $\mathcal{B}$, the reduction of the sentence $A$ to the game induction principle is proved.
\end{proof}
Using this generalization it is trivial to reprove the case for Buss's hierarchy of bounded arithmetic:
\begin{cor}\label{t3-12}(\cite{ST}, \cite{Ta})
For all $j \leq k$, $\forall \hat{\Sigma}^b_j(T^k_2) \equiv GI_k^j$.
\end{cor}
Now, let us explain the second type of problems, i.e., the generalized local search problems:
\begin{dfn}\label{t3-13}
A formalized $(\Psi, \Lambda, \mathcal{B}, \prec, t)$-GLS problem consists of the following data:
\begin{itemize}
\item[$(i)$]
A term $N(x, s) \in \mathcal{L}_{\mathcal{B}}$ as local improvement.
\item[$(ii)$]
A term $c(x, s) \in \mathcal{L}_{\mathcal{B}}$ as cost function.
\item[$(iii)$]
A predicate $F(x, s) \in \Psi$ which intuitively means that $s$ is a feasible solution for the input $x$.
\item[$(iv)$]
An initial term $i(x) \in \mathcal{L}_{\mathcal{B}}$.
\item[$(v)$]
A goal predicate $G(x, s) \in \Lambda$.
\item[$(vi)$]
A quantifier-free predicate $\prec \in \mathcal{L}_{\mathcal{B}}$ as a well-ordering.
\item[$(vii)$]
A bounding term $t(x)$.
\end{itemize}
such that $\mathcal{B}$ proves that $\prec$ is a total order and
\[
\mathcal{B} \vdash \forall x \; F(x, i(x))
\]
\[
\mathcal{B} \vdash \forall xs \; (F(x, s) \rightarrow F(x, N(x, s)))
\]
\[
\mathcal{B} \vdash \forall xs \; (N(x, s)=s \vee c(x, N(x, s)) \prec c(x, s))
\]
\[
\mathcal{B} \vdash \forall xs \; (G(x, s) \leftrightarrow (N(x, s)=s \wedge F(x, s)))
\]
\[
\mathcal{B} \vdash\forall xs \; (G(x, s) \rightarrow s \leq t(x))
\]
for some term $t$.\\
Moreover, if $\mathcal{L}_{\PV} \subseteq \mathcal{L}_{\mathcal{B}}$ and $t(x)=2^{p(|x|)}$ for some polynomial $p$ we show the $\GLS$ problem by $\PLS(\Psi, \Lambda, \prec, \mathcal{B})$ and if $F$ is quantifier-free in the language of $\mathcal{B}$, $G$ is quantifier-free in the language of $\PV$ we show the $\GLS$ problem by $\PLS(\prec, \mathcal{B})$. Finally if $\mathcal{B}=\PV$, then we write $\PLS(\prec)$.
\end{dfn}
\begin{thm}\label{t3-14}
If $A \in \Pi_k(\Phi)$, then $\mathfrak{B}(\mathbb{T}_{all}, \Pi_{k+1}(\Phi), \mathcal{B}) \vdash \forall x \exists y \leq t(x) A(x, y) $ iff the search problem of finding $y$ by $x$ is reducible by a projection to an instance of a $\GLS(\Pi_k(\Phi), \{A\}, \mathcal{B}, \leq, t)$ provably in $\mathcal{B}$.
\end{thm}
\begin{proof}
Assume 
\[
\mathfrak{B}(\mathbb{T}_{all}, \Pi_{k+1}(\Phi), \mathcal{B}) \vdash \forall x \exists y \leq t(x) A(x, y).
\]
Then, we know that $\forall y \leq t(x) \neg A(x, y) \Rightarrow \bot$ is provable in the theory. By soundness theorem \ref{t2-30}, there exist a term $s(x)$, a formula $H(u, x) \in \Pi_{k+1}(\Phi)$ and sequences of terms $E_0$, $E_1$, $G_0$, $G_1$ and $F(u)$ such that the following statements are provable in $\mathcal{B}$:
\begin{itemize}
\item[$(i)$]
$H(0, \vec{x}) \equiv_{d}^{(E_0, E_1)} \forall y \leq t(\vec{x}) \neg A(\vec{x}, y)$.
\item[$(ii)$]
$H(t(x), \vec{x})\equiv_{d}^{(G_0, G_1)} \bot$.
\item[$(iii)$]
$\forall u < t(x) \; H(u, \vec{x}) \leq_{d}^{F_u} H(u+1, \vec{x})$.
\end{itemize} 
Since $H \in \Pi_{k+1}(\Phi)$, we have $H(u, x)=\forall \vec{v} \leq \vec{r}(\vec{x}, u)G(u, \vec{v}, \vec{x})$ where $G(\vec{v}, u, x) \in \Sigma_k(\Phi)$. Use the deterministic reductions to show the existence of terms $U$, $V$ and $Z$ such that
\begin{itemize}
\item[$(i)$]
$\mathcal{B} \vdash A(Z(\vec{v}), x) \rightarrow G(0, \vec{v}, \vec{x})$.
\item[$(ii)$]
$\mathcal{B} \vdash G(t(x), \vec{U}, \vec{x} ) \rightarrow \bot$.
\item[$(iii)$]
$\mathcal{B} \vdash \forall u < t(x) G(u, \vec{V}(u, \vec{v}, \vec{x}), \vec{x}) \rightarrow G(u+1, \vec{v}, \vec{x})$.
\end{itemize} 
Now define 
\[
F(x, u, \vec{v}, z)=
\begin{cases}
\neg G(u-1, \vec{v}) & u > 0  \\
z \leq t \wedge A(z, x) & u=0 \\
\end{cases} 
\]
and
\[
N(x, u, \vec{v}, z)=
\begin{cases}
(u-1, \vec{V}(u, \vec{v}, \vec{x}), z) & u > 1  \\
(0, \vec{v}, Z(\vec{v})) & u=1 \\
(u, \vec{v}, z) & u=0 \\
\end{cases} 
\]
and $Goal(x, u, \vec{v}, z)=[z \leq t \wedge A(x, z)]$, $i(x)=(t(x), \vec{U}, 0)$, and $c(u, \vec{v})=u$. It is clear to see that this data is a $(\Pi_k(\Phi), \{A\}, \mathcal{B}, \leq, t) $-$\GLS$ problem. The answer to this problem is $(0, u, \vec{v}, z)$ where $A(z, x)$ holds. Note that by a projection we can extract $z$ from it which is the witness for $\exists y$ in $A$.
\end{proof}
Again we have the special case for Buss's hierarchy:
\begin{cor}\label{t3-15}(\cite{BB})
For all $l \leq k$, $\forall \Sigma_{l+1}^b(T^{k+1}_2) \equiv \PLS(\Pi^b_k, \Pi^b_l, \PV, \leq)$.
\end{cor}
\begin{rem}
Note that the power of characterizations via these kinds of problems are more limited than the theory of flows'. The reason is that the game induction method relaxes the condition of provability of reductions and $\GLS$ problems unwind just one universal quantifier and put the rest into the feasibility predicate.
\end{rem}
Using this characterization by $\GLS$ problems, we can capture the class of total $\NP$ search problems in strong theories:
\begin{cor}\label{t3-16}
$\TFNP(I\Delta_0+\EXP)\equiv \PLS(\mathcal{R}(exp), \leq) $.
\end{cor}
\begin{lem}\label{t3-17}
$\TFNP(\PRA_{\prec})\equiv \PLS(\BASIC_{\prec}, \leq) $.
\end{lem}
Therefore by definition of $\Pi^0_2$-ordinal and the fact that $\PRA_{\prec}$ is a conservative extension of $\PRA+\PRWO(\prec)$, we have:
\begin{thm}\label{t3-18}
Let $T$ be a theory of arithmetic with $\Pi^0_2$-ordinal $\alpha_T$ with a $\PRA$-representation $\prec_{\alpha_T}$, then $\TFNP(T) \equiv \PLS(\BASIC_{\prec_{\alpha_T}}, \leq)$.
\end{thm}
And finally by Theorem \ref{t3-8} we have:
\begin{cor}\label{t3-19}
\begin{itemize}
\item[$(i)$]
$\TFNP(I\Sigma_1) \equiv \PLS(\BASIC_{\prec_{\omega^2}}, \leq)$.
\item[$(ii)$]
For all $n > 1$, $\TFNP(I\Sigma_n) \equiv \PLS(\BASIC_{\prec_{\omega_n}}, \leq)$.
\item[$(iii)$]
$\TFNP(\PA)\equiv \PLS(\BASIC_{\prec_{\epsilon_0}}, \leq)$.
\item[$(iv)$]
For any $\PRA$-representable ordinal $\epsilon_0 \prec \alpha $, $\TFNP(\PA+\TI(\alpha)) \equiv \PLS(\BASIC_{\prec_{\alpha}}, \leq)$.
\end{itemize}
\end{cor}
\begin{rem}
These characterizations of search problems of strong theories of arithmetic may seem a bit counter-intuitive. The reason is as follows: Assume that we are working with $I\Sigma_1$. Then by Corollary \ref{t3-19}, we have access to all primitive recursive functions and predicates for our formulas and reductions and what we want to solve is just an $\NP$ search problem. Hence, having this huge power, it seems that just one reduction should be enough and it means that our characterization is weak or trivial in some sense. This is not the case and the explanation is as follows: It is correct that we have access to all primitive recursive functions but they act just like oracles in a black box. We can ask our questions but we can not understand their behavior and hence we can not be sure about the truth of their answers. Therefore, we need to use a long sequence of reductions and in each reduction we can be sure of the very limited part of the argument. But if you still think that this incomplete access to complex functions is unbearable even with the mentioned explanation, we will refer you to the next section in which we eliminate the presence of complex function symbols via proof theoretic ordinals.
\end{rem}
In the rest of this section we will explain some applications of non-deterministic flows. But first of all let us explain why we need this kind of non-determinism. Assume that we are working in the theory $S^k_2$ which has the polynomial induction and not the usual one. If we want to decompose proofs of this theory to a sequence of reductions, we have to kill the effect of the contraction rule. But simulating contraction needs an exponential sequence of reductions which we can not afford by our polynomial induction. Hence in this situation and in all the situations that the induction is extremely weaker than the bounds of the formulas, it is natural to work with reductions that handle the contraction rule automatically, and this power is exactly what the non-deterministic reductions provide.\\

To apply the non-deterministic soundness, let us first define a hierarchy of theories of bounded arithmetic to have a variety of theories with gaps between term bounds and induction lengths:
\begin{dfn}\label{t3-20}
Define $\mathbb{T}_m$ as the term ideal consisting of all terms less than terms of the form $|t|_{m}$. Define the theory $R^k_{m, n}$ as $\mathfrak{B}(\mathbb{T}_m, \Pi^b_k(\#_n), \BASIC(\#_n))$. 
\end{dfn}
In the following theorem, we show that it is possible to decompose proofs of $R^k_{m, n}$:
\begin{thm}\label{t3-21}
Let $\Gamma, \Delta \subseteq \Pi^b_k(\#_n)$, then $R^k_{m, n} \vdash \Gamma \Rightarrow \Delta$ iff 
\[
\Gamma \rhd_n^{(\mathbb{T}_m, \Pi^b_k(\#_n), \BASIC(\#_n))} \Delta.
\]
\end{thm}
The previous theorem is useful for some specific cases that we are interested in. For the first application, we can reprove some strong version of Buss's witnessing theorem for the $\{S_2^k\}_{k=0}^{\infty}$ hierarchy:
\begin{cor}\label{t3-22}(Strong Witnessing Theorem)
The provably $\Sigma^b_k$-definable functions of $S_2^k$ are in $\Box_{k}^p$, provably in $\PV$, i.e. if $S^k_2 \vdash \forall \vec{x} \exists y A(\vec{x}, y)$ where $A(\vec{x}, y) \in \Sigma^b_k$, then there exist a machine $M$ computing a function $f \in \Box_k^p$ and polytime function symbol $g$ such that $\PV \vdash comp_M(\vec{x}, w) \rightarrow A(\vec{x}, g(\vec{x}, w))$.
\end{cor}
\begin{proof}
Assume $S^k_2 \vdash \forall \vec{x} \exists y A(\vec{x}, y)$. By Parikh theorem we know that there exists a bound for the existential quantifier. Hence $S^k_2 \vdash \forall y \leq t(x) \; \neg A(\vec{x}, y) \Rightarrow \bot$. By Theorem \ref{t3-21} we know that there exist a polynomial $p(\vec{|x|})$ and a formula $H(u, \vec{x}) \in \Pi^b_k$ such that the following statements are provable in $\mathcal{R} \cup \BASIC$:
\begin{itemize}
\item[$(i)$]
$H(0, \vec{x}) \leftrightarrow \forall y \leq t(x) \; \neg A(\vec{x}, y)$.
\item[$(ii)$]
$ H(p(|\vec{x}|), \vec{x}) \leftrightarrow \bot$.
\item[$(iii)$]
$\forall u < p(\vec{|x|}) \; H(u, \vec{x}) \rightarrow H(u+1, \vec{x})$.
\end{itemize}
Note that $H(u, \vec{x})=\forall \vec{z} \leq s(\vec{x}) \; G(u, \vec{x}, \vec{z})$ where $G(u, \vec{x}, \vec{z}) \in \Sigma^b_{k-1}$. Since $\mathcal{R}\cup \BASIC$ is a universal theory, by the generalization of Herbrand's theorem we know that there exists a $\vee$-expansion of formulas $ \forall y \leq t(x) \; \neg A(x, y) \rightarrow H(0, \vec{x})$, $ H(t(x), \vec{x}) \rightarrow \bot$ and $\forall u < p(\vec{|x|}) \; H(u, \vec{x}) \rightarrow H(u+1, \vec{x})$ such that we can witness existential quantifiers by terms. Note that since we have the power to decide all formulas in $\Sigma^b_{k-1}$, we can kill the effect of the expansion to find the polytime functions to witness the existential quantifiers such that:
\begin{itemize}
\item[$(i)$]
$(U(\vec{x}, z) \leq t(\vec{x}) \rightarrow \neg A(\vec{x}, U(\vec{x}, \vec{z})) \rightarrow  G(0, \vec{x}, \vec{z})$.
\item[$(ii)$]
$ G(t(x), \vec{x}, \vec{V}) \rightarrow \bot$.
\item[$(iii)$]
$\forall u < p(\vec{|x|}) \; G(u, \vec{x}, \vec{Z}(\vec{x}, \vec{z})) \rightarrow G(u+1, \vec{x}, \vec{z})$.
\end{itemize}
Now, define the algorithm $M$ as the following: Begin with $\vec{V}$ and do the following for $p(|x|)$ many steps: In each step apply $\vec{Z}$, write it somewhere and ask the oracle about $G(u, \vec{x}, \vec{Z}(\vec{x}, \vec{z}))$ and save it also somewhere else.\\

We claim that this $M$ works. If we have the whole computation of $M$, i.e. $w$, it is easy to compute the witness in that step, $\vec{a}_u$ and value of $G(u, \vec{x}, \vec{a}_u)$ by poly-time functions $\vec{v}(w, u)$ and $j(w, u, \vec{v})$ provably in $\PV$. Hence, the statement $j(w, u, \vec{v}(w, u))=0$ is provable by length induction on $u$ and therefore provably in $\PV$ we know that if $w$ is the computation of $M$ then $j(w, p(|x|), \vec{v}(w, p(|x|)))=0$ and thus $\neg G(0, \vec{x}, \vec{v}(w, p(|x|)))$ and hence $\vec{v}(w, p(|x|)) \leq \vec{t}(x)$ and $A(x, \vec{v}(w, p(|x|))$. Pick $g(w)=\vec{v}(w, p(|x|))$ and we have the claim.
\end{proof}

As the second application, note that if we put $m=n-2$, $\Gamma=\emptyset$ and $\Delta=\{\forall x \exists y \leq |t(x)|_{n-2} \; A(x, y)\}$ where $A$ is $(n-1)$-bounded, the previous theorem in the presence of RSUV isomorphism, finds a way to extract the information about $\NP$ search problems of higher-order bounded arithmetic expressed in the first order language by using faster growing smash functions.

\section{Ordinal Flows}
In the previous sections we investigated bounded theories of arithmetic and we proved that they are sound and complete with respect to their appropriate flow-based interpretations. Now, it is natural to seek for a similar theory for unbounded theories of arithmetic. First of all, note that since we are interested in low-complexity statements and since it is possible to reduce the whole quantifier complexity of unbounded strong enough theories to universal statements via their proof theoretic ordinals, it is natural to restrict our investigations to these universal theories with ordinal induction for universal formulas. Note that here we are not working with bounded theories and hence assuming that the length of a flow is a term seems inappropriate. But clearly, there is also a natural candidate in this case, which is the proof theoretic ordinal. Therefore, in the case of strong enough unbounded theories we will work with flows of universal formulas with ordinal length and we will use them to extract the computational information of the theories.  

\begin{dfn}\label{t4-1}
Let $\mathcal{L}_{\PV}$ be the language of $\PV$. Define the system $\TI(\forall_1, \prec)$ as the usual first order sequent calculus of first order language plus the axioms of $\PV$ and the following induction rule:
\begin{center}
	\begin{tabular}{c}
	    \AxiomC{$\Gamma, \forall \gamma \prec \beta \; A(\gamma) \Rightarrow \Delta, A(\beta)$}
		\LeftLabel{\tiny{$ (Ind_\alpha) $}}
		\UnaryInfC{$\Gamma \Rightarrow \Delta, A(\delta)$}
		\DisplayProof
	\end{tabular}
\end{center}		
For every $A \in \forall_1$ where $\forall_1$ means the class of all universal formulas.
\end{dfn}
Using $\Pi^0_2$-ordinal we can transfer $\Pi^0_2$ sentences form a theory $T$ to the theory $\PRA+\PRWO(\prec)$ where $\prec$ is a $\PRA$-representation of $\alpha_T$. The following theorem makes it possible to continue this process of transferring to $\TI(\forall_1, \prec)$ which is a more convenient theory for our technical purpose. 
\begin{lem}\label{t4-2}
$\PRA+\PRWO(\prec) \subseteq \TI(\forall_1, \prec)$.
\end{lem}
\begin{proof}
First of all, notice that it is possible to represent any primitive recursive function $f$ by a polynomial time computable predicate $F$. We will use this definition to interpret all quantifier-free statements in $\PRA$ as formulas in $\forall_1$ statements in the language of $\PV$. By our way of interpretation the defining axioms in $\PRA$ are provable in $\TI(\forall_1, \prec)$. For the induction, it is enough to use induction on $\omega \prec \alpha$ in $\TI(\forall_1, \prec)$. What remains is the axiom $\PRWO(\prec)$. Note that the interpretation of this axiom is $\forall y \forall uv \; (F(\vec{z}, y+1, u) \wedge F(\vec{z}, y, v)) \rightarrow u \prec v \Rightarrow \bot$. We know that $f(0)$ exists, i.e. $\forall a \neg F(0, a) \Rightarrow \bot$. To prove, use induction on
\[
A(x)=
\begin{cases}
\bot  & F(r(x), q(x)) \wedge x \prec a \omega \\
\top & o.w.\\
\end{cases} 
\] 
where $q(x)=\lfloor \frac{x}{\omega} \rfloor $ and $r(x)=x-q(x)$.\\ 

$A$ is inductive because if $\forall z \prec x \; A(x)$ is true and $A(x)$ is false, then by definition $x \prec \omega a$ and $F(q(x), r(x))$. Pick $c$ as $f(r(x)+1)$ which we know exists. Therefore by the assumption we know $\forall y \forall uv \; F(\vec{z}, y+1, u) \wedge F(\vec{z}, y, v) \rightarrow u \prec v$ and hence $c \prec f(r(x))$. If $a \neq 0$ we have $a(f(r(x)+1))+r(x)+1 \prec af(r(x))+r(x)=x$. Hence $A(a(f(r(x)+1))+r(x)+1)$ is $\bot$ which contradicts $\forall z \prec x \; A(x)$. If $a=0$ then $f(0)=0$ which contradicts $f(1) \prec f(0)=0$.
\end{proof}
We have defined our theory so far. Let us now define the concept of ordinal flows.
\begin{dfn}\label{t4-3}
Let $A(\vec{x})$, $B(\vec{x})$ and $H(\delta, \vec{x})$ be some formulas in $\forall_1$. A tuple $(H, \beta)$ is called an $\alpha$-flow if
\begin{itemize}
\item[$(i)$]
$\PV \vdash A(\vec{x}) \rightarrow \forall \gamma \prec H(1, \vec{x})$.
\item[$(ii)$]
$\PV \vdash \forall \gamma \prec \beta \; [\forall \delta \prec \gamma \; H(\delta, \vec{x}) \rightarrow \forall \delta \prec \gamma+1 \; H(\delta, \vec{x})]$.
\item[$(iii)$]
$\PV \vdash \forall \gamma \prec \beta \; H(\gamma, \vec{x}) \rightarrow B(\vec{x})$.
\end{itemize} 
\end{dfn}
Like the bounded case we need to prove some basic theorems for this new notion. They will help us to prove the soundness theorem for this kind of flow.
\begin{lem}\label{t4-4}(Conjunction Application)
Let $C(\vec{x}) \in \forall_1$ be a formula. If $A(\vec{x}) \rhd B(\vec{x}) $ then $A(\vec{x}) \wedge C(\vec{x}) \rhd B(\vec{x}) \wedge C(\vec{x}) $.
\end{lem}
\begin{proof}
Since $A(\vec{x}) \rhd B(\vec{x}) $, then by Definition \ref{t4-3} there exist a term $\beta$ and a formula $H(\gamma, \vec{x}) \in \forall_1$ such that we have the conditions in the Definition \ref{t4-3}. Define $\beta'=\beta$ and $H'(\gamma, \vec{x})=H(\gamma, \vec{x}) \wedge C(\vec{x})$. It is clear that the $(H', \beta')$ is an $\alpha$-flow from $A(\vec{x}) \wedge C(\vec{x})$ to $B(\vec{x}) \wedge C(\vec{x})$.
\end{proof}
\begin{lem}\label{t4-5}(Disjunction Application)
Let $C(\vec{x}) \in \forall_1$ be a formula. If $A(\vec{x}) \rhd B(\vec{x}) $ then $A(\vec{x}) \vee C(\vec{x}) \rhd B(\vec{x}) \vee C(\vec{x}) $.
\end{lem}
\begin{proof}
Since $A(\vec{x}) \rhd B(\vec{x}) $, then by Definition \ref{t4-3}, there exist an ordinal $\beta$ and a formula $H(\gamma, \vec{x}) \in \forall_1$ such that the conditions in the Definition \ref{t4-3} is provable in $\PV$. Now define $\beta'=\beta$ and $H'(\gamma, \vec{x})=H(\gamma, \vec{x}) \vee C(\vec{x})$. It is easy to see that $(H', \beta')$ is an $\alpha$-flow from $A(\vec{x}) \vee C(\vec{x})$ to $B(\vec{x}) \vee C(\vec{x})$.   
\end{proof}
\begin{lem}\label{t4-7}
\begin{itemize}
\item[$(i)$](Weak Gluing)
If $A(\vec{x}) \rhd B(\vec{x}) $ and $ B(\vec{x}) \rhd C(\vec{x})$, then $A(\vec{x}) \rhd C(\vec{x})$.
\item[$(ii)$](Strong Gluing)
If $ \forall \gamma \prec \beta A(\gamma, \vec{x}) \rhd \gamma \prec \beta+1 A(\gamma, \vec{x})$, then $\top \rhd A(\theta, \vec{x})$.
\end{itemize}
\end{lem}
\begin{proof}
For $(i)$, since $A(\vec{x}) \rhd B(\vec{x})$ there exist an ordinal $\beta$ and a formula $H(\gamma, \vec{x}) \in \forall_1$ such that $\PV$ proves the conditions in the Definition \ref{t4-3}. On the other hand since $B(\vec{x}) \rhd C(\vec{x})$ we have the corresponding data for $B(\vec{x})$ to $C(\vec{x})$ which we show by $\beta'$ and $H'(\gamma, \vec{x})$. Define $\beta''=\beta+\beta'$ and
\[
H''(\gamma, \vec{x})=
\begin{cases}
H(\gamma, \vec{x}) & \gamma \preceq \beta \\
H'(\gamma-\beta, \vec{x}) & \beta \prec u \preceq \beta+\beta'
\end{cases} 
\]
It is easy to check that $(\beta'', H'')$ is an $\alpha$-flow from $A(\vec{x})$ to $C(\vec{x})$. \\

For $(iii)$ if we have $\forall \gamma \prec \delta A(\gamma, \vec{x}) \rhd \forall \gamma \prec \delta+1 A(\gamma, \vec{x})$ then there exists $\beta$ and $H(\gamma, \delta, \vec{x})$ such that we have the conditions of the Definition \ref{t4-3}. Define $\beta'=\beta \times \theta$ and $I(\gamma, \vec{x})=
H(\lfloor \frac{\gamma}{\theta} \rfloor, \lfloor \frac{\gamma}{\theta} \rfloor, \vec{x})$.
It is easy to see that $(I, \beta')$ is an $\alpha$-flow from $\top$ to $A(\theta, \vec{x})$.
\end{proof}
\begin{lem}\label{t4-6}(Conjunction and Disjunction Rules)
\begin{itemize}
\item[$(i)$]
If $\Gamma, A \rhd \Delta$ or $\Gamma, B \rhd \Delta$, then $\Gamma, A \wedge B \rhd \Delta$.
\item[$(ii)$]
If $\Gamma_0 \rhd \Delta_0, A $ and $\Gamma_1 \rhd \Delta_1, B$, then $\Gamma_0, \Gamma_1 \rhd \Delta_0, \Delta_1, A \wedge B$.
\item[$(iii)$]
If $\Gamma \rhd \Delta, A$ or $\Gamma \rhd \Delta, B$, then $\Gamma \rhd \Delta, A \vee B$.
\item[$(iv)$]
If $\Gamma_0, A \rhd \Delta_0$ and $\Gamma_1, B \rhd \Delta_1$, then $\Gamma_0, \Gamma_1, A \vee B \rhd \Delta_0, \Delta_1$.
\end{itemize}
\end{lem}
\begin{proof}
The proof is similar to the proof of the theorem \ref{t2-23}. Note that the proof of the theorem \ref{t2-23} is fully based on the weak gluing and conjunction and disjunction applications, hence we can apply the same proof wherever we have those properties.
\end{proof}

\begin{thm}\label{t4-8}(Soundness)
If $\Gamma \cup \Delta \subseteq \forall_1$ and $\TI(\forall_1, \prec) \vdash \Gamma \Rightarrow \Delta$, then there exists an $\alpha$-flow from $\Gamma$ to $\Delta$.
\end{thm}
\begin{proof}

We prove the lemma by induction on the length of the free-cut free proof of $\Gamma(\vec{x}) \Rightarrow \Delta(\vec{x})$.\\

1. (Axioms). If $\Gamma(\vec{x}) \Rightarrow \Delta(\vec{x})$ is a logical axiom then the claim is trivial. If it is a non-logical axiom then the claim will be also trivial because all non-logical axioms are provable in $\PV$. Therefore there is nothing to prove.\\

2. (Structural Rules). The case for weakening and exchange are trivial. For the contraction, note that all formulas are $\forall_1$ which means that having all quantifiers, it is possible to decide in polynomial-time which formula is true and hence we can handle the contraction case. \\ 

3. (Cut). It is similar to the Lemma \ref{t2-28}.\\

4. (Propositional). The conjunction and disjunction cases are proved in the Lemma \ref{t4-6}. The implication and negation cases are trivial because they should be quantifier-free and hence we can manipulate them as in the Lemma \ref{t2-26} and \ref{t2-29}.\\

5. (Universal Quantifier, Right). If $\Gamma(\vec{x}) \Rightarrow \Delta(\vec{x}), \forall z  B(\vec{x}, z)$ is proved by the $\forall R$ rule by $\Gamma(\vec{x}) \Rightarrow \Delta(\vec{x}), B(\vec{x}, z)$, then by IH, $\Gamma(\vec{x}) \rhd \Delta(\vec{x}), B(\vec{x}, z)$. Therefore, there exist an ordinal $\beta$ and a formula $H(\gamma, \vec{x}, z) \in \forall_1$ such that the conditions of the Definition \ref{t4-3} are provable in $\PV$. Define $ \beta'=\beta$ and $H'(\gamma, \vec{x})=\forall z H(\gamma, \vec{x}, z)$. Since $H(\gamma, \vec{x}, z) \in \forall_1$ then $\forall z H(\gamma, \vec{x}, z) \in \forall_1$. The other conditions to check that the new sequence is an $\alpha$-flow is a straightforward consequence of the fact that if 
\[
\PV \vdash \forall \gamma \prec \delta H(\gamma, z, \vec{x}) \rightarrow \gamma \prec \delta+1 \; H(\gamma, z, \vec{x}),
\]
then
\[
\PV \vdash \forall \gamma \prec \delta \forall z H(\gamma, z, \vec{x}) \rightarrow \forall \gamma \prec \delta+1 \forall z H(\gamma, z, \vec{x}).
\]
\\
6. (Universal Quantifier, Left). If $\Gamma(\vec{x}), \forall z B(\vec{x}, z) \Rightarrow \Delta(\vec{x})$ is proved by the $\forall L$ rule by $\Gamma(\vec{x}), B(\vec{x}, s(\vec{x})) \Rightarrow \Delta(\vec{x})$, then since $\PV \vdash \forall z B(\vec{x}, z) \rightarrow B(\vec{x}, s(\vec{x}))$, we have
\[
\forall z B(\vec{x}, z) \leq B(\vec{x}, s(\vec{x})).
\]
And since
\[
\Gamma(\vec{x}), B(\vec{x}, s(\vec{x})) \rhd \Delta(\vec{x}),
\]
by using cut we have
\[
\Gamma(\vec{x}), \forall z B(\vec{x}, z) \rhd \Delta(\vec{x}).
\]
\\
7. (Induction). The proof is similar to the proof of Lemma \ref{t2-28}.

\end{proof}
And also like in the bounded case we have the completeness theorem:
\begin{thm}\label{t4-9}(Completeness)
If $\Gamma \cup \Delta \subseteq \forall_1$ and $\Gamma \rhd \Delta$, then $\TI(\forall_1, \prec) \vdash \Gamma \Rightarrow \Delta$.
\end{thm}
\begin{proof}
If there exists an $\alpha$-flow from $\Gamma$ to $\Delta$ then it means that there exists $(H, \beta)$ such that 
\begin{itemize}
\item[$(i)$]
$\PV \vdash A(\vec{x}) \rightarrow \forall \gamma \prec H(1, \vec{x})$.
\item[$(ii)$]
$\PV \vdash \forall \gamma \prec \beta \; [\forall \delta \prec \gamma \; H(\delta, \vec{x}) \rightarrow \forall \delta \prec \gamma+1 \; H(\delta, \vec{x})]$.
\item[$(iii)$]
$\PV \vdash \forall \gamma \prec \beta \; H(\gamma, \vec{x}) \rightarrow B(\vec{x})$.
\end{itemize} 
Therefore, using induction on $ H(\delta, \vec{x})$ we have 
\[
\TI(\forall_1, \prec) \vdash H(0, \vec{x}) \rightarrow H(\gamma, \vec{x}).
\]
And hence
\[
\TI(\forall_1, \prec) \vdash H(0, \vec{x}) \rightarrow \forall \gamma \prec \beta \; H(\gamma, \vec{x}),
\]
and thus $\TI(\forall_1, \prec) \vdash A(\vec{x}) \Rightarrow B(\vec{x})$.
\end{proof}
In the following we will use the $\PLS(\prec_{\alpha})$ problems to characterize the $\NP$ search problems of any theory with $\Pi^0_2$-ordinal $\alpha$.
\begin{thm}\label{t4-10}
Let $T$ be a theory of arithmetic and $\alpha_T$ be its $\Pi^0_2$-ordinal with a $\PV$-representation $\prec_{\alpha_T}$ of the order and a $\PV$-representation of its ordinal arithmetic, then $\TFNP(T)\equiv_{\PV} \PLS(\prec_{\alpha_T})$.
\end{thm}
\begin{proof}
First of all, it is easy to see that $\exists s \; \neg c(N(x, s)) \prec c(x, s) \wedge F(x, s)$ is provable in $\PRA+\PRWO(\prec)$. Define $f(0)=(c(x, i(x)), i(x))$ and
\[
f(n+1)=
\begin{cases}
(c(x, N(x, f_0(n))), N(x, f_0(n))) & c(x, N(x, f_0(n)) \prec c(x, f_0(n)) \wedge F(x, f_0(n)) \\
f(n) & o.w. \\
\end{cases} 
\]
where the order on the range of $\prec'$ is the order of the ordinal $ot(\prec_{c(x, i(x))}) \times \omega$. Since $\prec'$ is a sub-order of $\prec$, by $\PRWO(\prec)$ there exists some $n$ such that $f(n+1) \nprec f(n)$. By definition of $f$, this $n$ should impose the property that $c(x, N(x, f_0(n)) \nprec c(x, f_0(n)) \vee \neg F(x, f_0(n))$. It is easy to show by induction on $m$ that $F(x, f_0(m))$ for any $m$, hence $c(x, N(x, f_0(n)) \nprec c(x, f_0(n)) \wedge F(x, f_0(n))$. Now it is enough to pick $s=f_0(n)$. Therefore, $\PRA+\PRWO(\prec) \vdash \exists s \; N(x, s)=s \wedge F(x, s)$ and therefore, $\PRA+\PRWO(\prec) \vdash \exists s \; G(x, s)$. And finally, $\PRA+\PRWO(\prec) \vdash \exists s \; |s| \leq p(|x|) \wedge G(x, s)$ which by definition means $T  \vdash \exists s \; |s| \leq p(|x|) \wedge G(x, s)$.\\

For the converse, assume that $T \vdash \forall x \exists y |y| \leq p(|x|) A(x, y)$ where $A(x, y)$ is quantifier-free in the language of $\PV$. Then by definition $\PRA+\PRWO(\prec) \vdash \forall x \exists y |y| \leq p(|x|) A(x, y)$ because $\forall x \exists y \; |y| \leq p(|x|) \wedge A(x, y) \in \Pi^0_2$. Then by Lemma \ref{t4-2} we have 
\[
\TI(\forall_1, \prec_{\alpha_T}) \vdash \forall y (|y| \leq p(|x|) \rightarrow \neg A(x, y)) \Rightarrow \bot.
\] 
By Theorem \ref{t4-8}  we have $\forall y (|y| \leq p(|x|) \rightarrow \neg A(x, y)) \rhd \bot$. Hence there exists $(H, \beta)$ such that 
\begin{itemize}
\item[$(i)$]
$\PV \vdash \forall y (|y| \leq p(|x|) \rightarrow \neg A(x, y)) \rightarrow \forall \gamma \prec 1 \; H(1, \vec{x})$.
\item[$(ii)$]
$\PV \vdash \forall \gamma \prec \beta \; [\forall \delta \prec \gamma \; H(\delta, \vec{x}) \rightarrow \forall \delta \prec \gamma+1 \; H(\delta, \vec{x})]$.
\item[$(iii)$]
$\PV \vdash \forall \gamma \prec \beta \; H(\gamma, \vec{x}) \rightarrow \bot$.
\end{itemize} 
Since $H \in \forall_1$ we have $H(\gamma, x)=\forall z G(\gamma, x, z)$. On the other hand, all the conditions are provable in $\PV$ which means that we can witness the existential quantifiers by polytime functions. Hence,
\begin{itemize}
\item[$(i')$]
$\PV \vdash (|Y(x, z)| \leq p(|x|) \rightarrow \neg A(x, Y(x, z))) \rightarrow G(0, \vec{x}, z)$.
\item[$(ii')$]
$\PV \vdash \forall \gamma \prec \beta \; [\Delta(\delta) \prec \gamma \rightarrow G(\Delta(\delta), \vec{x}, Z(\delta)) \rightarrow \delta \prec \gamma+1  \rightarrow G(\delta, \vec{x}, z)]$.
\item[$(iii')$]
$\PV \vdash  (\Gamma \prec \beta \rightarrow  G(\Gamma, \vec{x}, Z)) \rightarrow \bot$.
\end{itemize}
Put $\delta=\gamma$ in $(ii')$, then we have 
\[
\PV \vdash \forall \gamma \prec \beta \; [(\Delta(\gamma) \prec \gamma \rightarrow G(\Delta(\gamma), \vec{x}, Z(\gamma)) \rightarrow  G(\gamma, \vec{x}, z)].
\]

Define $F(x, \gamma, y, z)=\neg G(x, \gamma, z)$ and 
\[
N(x, \gamma, y, z)=
\begin{cases}
(x, \Delta(\gamma), y, Z(\gamma)) & \gamma \neq 0, \neg G(x, \gamma, z) \\
(x, 0, y, 0) & \gamma \neq 0, G(x, \gamma, z) \\
(x, \gamma, y, z) & \gamma=0
\end{cases} 
\]
and $i(x)=(x, \Gamma, 0, Z)$ and 
$c(x, \gamma, y, z)=\gamma$, $Goal(x, \gamma, y, z)=G(x, 0, z)$. It is easy to see that this new data is a $\PLS(\prec_{\alpha_T})$ problem. Now it is not hard to shift everything for $\gamma \prec \omega$ one point to the right to add $|y| \leq p(|x|) \wedge A(x, y)$ to the first point and use $Y$ for its neighborhood. Now we have a $\PLS(\prec_{\alpha_T})$ problem and finally by the answer of the problem namely $(x, \gamma, y, z)$ we can compute $y$ which is the witness for $A$ and computable just by a projection. Note that this reduction is provable in $\PV$. 
\end{proof}
And as a corollary we have:
\begin{cor}\label{t4-11}
\begin{itemize}
\item[$(i)$]
$\TFNP(I\Sigma_1)\equiv \PLS(\prec_{\omega^2})$.
\item[$(ii)$]
For all $n > 1$, $\TFNP(I\Sigma_n)\equiv \PLS(\prec_{\omega_n})$.
\item[$(iii)$]
$\TFNP(\PA)\equiv \PLS(\prec_{\epsilon_0})$.
\item[$(iv)$]
For any representable $\epsilon_0 \prec \alpha $, $\TFNP(\PA+\TI(\alpha)) \equiv \PLS(\prec_{\alpha})$.
\end{itemize}
\end{cor}
\begin{proof}
It is enough to have a $\PV$-representation of these ordinals and the basic arithmetic on them which was carried out in \cite{BBP}.
\end{proof}

\vspace{4pt}
\textbf{Acknowledgment.} 
We wish to thank Pavel Pudlak for his support, his suggestions and the invaluable discussions that we have had since the beginning of this project. We are also genuinely grateful to Sam Buss and Raheleh Jalai for their constructive suggestions and the helpful discussions on the crucial and primitive stages of developing the theory.

\end{document}